\documentclass[preprint]{elsarticle}

\usepackage{lineno,hyperref}
\modulolinenumbers[5]

\usepackage{amsmath}
\usepackage{amssymb}
\usepackage[utf8]{inputenc}
\usepackage[T1]{fontenc}
\usepackage{enumerate}
\usepackage{a4wide}
\usepackage{tikz}

\def\NN{\mathbb N}
\def\bb{\mathcal B}

\def\uu{\mathcal U}

\def\L{\mathcal{L}}
\def\uuu{\overline{\mathcal U}}

\def\vv{\mathcal V}

\def\k{\mathbf{k}}

\def\s{\mathbf{s}}

\def\j{\mathbf{j}}
\def\f{\infty}
\def\ra{\rightarrow}
\def\de{\delta}

\def\c{\mathbf{c}}
\def\d{\mathbf{d}}

\def\al{\alpha}

\def\si{\sigma}
\def\C{\mathcal{C}}
\def\D{\mathcal{D}}

\newcommand{\set}[1]{\left\lbrace #1\right\rbrace}

\newcommand{\remove}[1]{ }
\newcommand{\qtq}[1]{\quad\text{#1}\quad}

\DeclareMathOperator{\dist}{dist}

\newtheorem{theorem}{Theorem}[section]
\newtheorem{proposition}[theorem]{Proposition}
\newtheorem{lemma}[theorem]{Lemma}
\newtheorem{corollary}[theorem]{Corollary}

\newtheorem{remark}[theorem]{Remark}
\newtheorem{remarks}[theorem]{Remarks}
\newtheorem{examples}[theorem]{Examples}
\newtheorem{example}[theorem]{Example}

\newproof{proof}{Proof}

\numberwithin{equation}{section}

\usepackage[english]{babel}

\begin{document}
\begin{frontmatter}
\title{Bases in which some numbers have exactly two expansions}

\author{Vilmos Komornik}
\address{D\'{e}partement de math\'{e}matique,
         Universit\'{e} de Strasbourg,
         7 rue Ren\'{e} Descartes,
         67084 Strasbourg Cedex, France}
\ead{komornik@math.unistra.fr}

\author{Derong Kong\corref{cor1}}
\address{Mathematical Institute, University of Leiden, PO Box 9512, 2300 RA Leiden, The Netherlands}
\cortext[cor1]{Corresponding author}
\ead{d.kong@math.leidenuniv.nl}


\begin{abstract}
In this paper we answer several questions raised by Sidorov on the set $\mathcal B_2$ of bases in which there exist numbers with exactly two expansions. In particular, we prove that the set $\mathcal B_2$ is closed, and it contains both infinitely many isolated and accumulation points in $(1, q_{KL})$, where $q_{KL}\approx 1.78723$ is the Komornik-Loreti constant. Consequently    we show that the second smallest element  of $\mathcal B_2$ is  the smallest accumulation point of $\mathcal B_2$. We also investigate the higher order derived sets of $\mathcal B_2$. Finally, we prove  that there exists a $\delta>0$ such that
\begin{equation*}
\dim_H(\mathcal B_2\cap(q_{KL}, q_{KL}+\delta))<1,
\end{equation*}
where $\dim_H$ denotes the Hausdorff dimension.
\end{abstract}


\end{frontmatter}

\section{Introduction}\label{s1}

In this paper we consider expansions
\begin{equation*}
((c_i))_q:=\sum_{i=1}^{\infty}\frac{c_i}{q^i}=x
\end{equation*}
over the \emph{alphabet} $\set{0,1}$ in some base $q>1$. The sequence $(c_i)=c_1c_2\cdots$ is called a \emph{$q$-expansion} of $x$. 
Such an expansion may exist only if $x\in I_q:=[0,1/(q-1)]$.

If $q>2$, then every expansion is unique by an elementary argument, but not every $x\in I_q$ has an expansion.
If $q\in (1,2]$, then every $x\in I_q$ has at least one expansion.
For example, in the well-known integer base case $q=2$ all numbers have a unique expansion, except the dyadic rational numbers that have two.
Henceforth we assume that $q\in (1,2]$.

Expansions in non-integer bases $q\in(1,2)$   are more complicated than  that in the integer base $q=2$.  Erd\H{o}s et al. discovered in \cite{Erdos_Joo_Komornik_1990, Erdos_Horvath_Joo_1991, Erdos_Joo_1992} that for any $k\in\mathbb N\cup\{\aleph_0\}$  there exist  $q\in(1,2)$ and $x\in I_q$ such that $x$ has precisely $k$ different $q$-expansions. 
Furthermore, in case $q\in (1,\varphi)$, where $\varphi\approx 1.61803$ denotes the \emph{Golden Ratio}, every $x\in (0,1/(q-1))$ has a continuum of  $q$-expansions.
Later, Sidorov proved in \cite{Sidorov_2003} that for any $q\in(1, 2)$ Lebesgue almost every $x\in I_q$ has a continuum of  $q$-expansions (see also \cite{Dajani_DeVries_2007}).
  
Many works have been devoted to unique expansions.
Komornik and Loreti proved in \cite{Komornik-Loreti-1998} that there is a smallest base $q_{KL}\approx 1.78723$ in which $x=1$ has a unique expansion. 
Subsequently, Glendinning and Sidorov discovered that the \emph{Komornik--Loreti constant} $q_{KL}$ plays an important role in describing the size of the \emph{univoque sets}
\begin{equation*}
\uu_q:=\{x\in I_q: x\textrm{ has a unique }q\textrm{-expansion}  \},\quad q\in(1,2].
\end{equation*}
They proved in \cite{Glendinning_Sidorov_2001} the following interesting results:
\begin{itemize}
\item if  $1<q\le \varphi$,  then  $\uu_q=\set{0, 1/(q-1)}$;
\item if $\varphi<q<q_{KL}$, then $\uu_q$ is countably infinite;
\item if $q=q_{KL}$, then $\uu_q$ is uncountable and has zero Hausdorff dimension;
\item if $q_{KL}<q\le 2$, then $\dim_H\uu_q>0$, where $\dim_H$ denotes the Hausdorff dimension. 
\end{itemize} 
We point out that $\uu_q$ has a fractal structure for $q\in(q_{KL}, 2)$. 
The authors and Li proved in  \cite{Komornik_Kong_Li_2015_1} that the function $q\mapsto \dim_H\uu_q$ has a Devil's staircase behavior.     
 
Based on the classification of the bases $q\in(1,2]$ by Komornik and Loreti \cite{Komornik_Loreti_2007}, the topology of the univoque sets $\uu_q$ was studied by de Vries and Komornik in  \cite{DeVries_Komornik_2008}. 
They proved among others that $\uu_q$ is closed if and only if $q\notin\overline{\uu}$, where $\overline{\uu}$ is the topological closure of
\begin{equation*}
\uu:=\{q\in(1,2]: 1\in\uu_q\}.
\end{equation*}
In other words, $\uu$ is the set of  \emph{univoque bases} $q\in(1,2]$ for which $1$ has a unique expansion. 
An element of $\uu$ is called a univoque base. 
As mentioned above, the smallest univoque base is $q_{KL}$. 

The set $\uu$ itself is also of  general  interest.  
Erd\H{o}s et al. proved in \cite{Erdos_Horvath_Joo_1991} that $\uu$ is a Lebesgue null set of first category.  Later, Dar\'{o}czy and K\'{a}tai proved in \cite{Darczy_Katai_1995} that $\uu$ has full Hausdorff dimension.  The  topological structure of $\uu$ was clarified in \cite{Komornik_Loreti_2007}. 
In particular,  the authors proved that $\overline{\uu}$ is a Cantor set, and $\overline{\uu}\setminus\uu$ is a countable and dense subset of  $\overline{\uu}$. 
Recently, Bonanno et al. investigated in \cite{Bonanno_Carminati_Isola_Tiozzo-2013} the connections between $\uu$, $\alpha$-continued fractions, unimodal maps and even the external rays of the Mandelbrot set.  

For more information on the sets $\uu_q$ and $\uu$ we refer to the survey paper \cite{Komornik_2011} and the book chapter \cite{deVries-Komornik-2016}. 

Sidorov initiated in \cite{Sidorov_2009} the study of the sets
\begin{equation*}
\bb_k:=\set{q\in(1,2]\,:\, \exists \; x\in I_q \textrm{ having precisely }k\textrm{ different }q\textrm{-expansions}}
\end{equation*}
for $k=2,3,\ldots$ and $k=\aleph_0$.
In particular, he has obtained the following important results for the set $\bb_2$. 

\begin{theorem}[Sidorov, 2009]\label{t11}\mbox{}
\begin{enumerate}[\upshape (i)]
\item $q\in\bb_2\Longleftrightarrow 1\in\uu_q-\uu_q$;
\item $\uu\subset\bb_2$;
\item $[\varphi_3,2]\subset\bb_2$, where $\varphi_3\approx 1.83929$
denotes the Tribonacci number, i.e., the positive zero of $q^3-q^2-q-1$;
\item the smallest two  elements   of $\bb_2$  are
\[q_s\approx 1.71064\qtq{and} q_f\approx 1.75488,\]
     the positive zeros   of  the polynomials  $q^4-2q^2-q-1$   and  $q^3-2q^2+q-1$ respectively.
\end{enumerate}
\end{theorem}
\noindent Theorem \ref{t11} (ii)  was also contained in \cite[Theorem 1.3]{DeVries_Komornik_2008}.

Until now very little is known about the sets $\bb_k$. 
Sidorov proved in \cite{Sidorov_2009} that $\bb_k$   contains a left neighborhood of $2$ for each $k\ge 2$.
Baker and Sidorov proved in \cite{Baker_Sidorov_2014} that the second smallest element $q_f$ of $\bb_2$ is the smallest element of  $\bb_k$ for each $k\ge 3$.  
It follows from the results of Erd\H os et al.  \cite{Erdos_Joo_Komornik_1990, Erdos_Horvath_Joo_1991} that the {Golden Ratio} $\varphi$ is the smallest element of $\bb_{\aleph_0}$. 
Recently, Baker  proved in \cite{Baker_2015} that $\bb_{\aleph_0}$ has a second smallest element which is strictly smaller than the smallest element  $q_s$ of $\bb_2$. 
Based on his work Zou and Kong proved in \cite{Zou_Kong_2015} that $\bb_{\aleph_0}$ is not closed. 

The purpose of this paper is to continue the investigations on the set $\bb_2$. 
We answer in particular the following questions of Sidorov \cite{Sidorov_2009}:
\begin{itemize}
\item[Q1.] Is $\bb_2$ a closed set?
\item[Q2.] Is $\bb_2\cap (1,q_{KL})$ a discrete set?
\item[Q3.] Is it true that
\begin{equation*}
\dim_H(\bb_2\cap(q_{KL}, q_{KL}+\de))<1
\end{equation*}
for some $\de>0$?
\end{itemize} 
Some ideas of this paper might be useful for the future study of $\bb_k$ with $k\ge 3$. 

Motivated by the work of de Vries and Komornik \cite{DeVries_Komornik_2008} (see also~\cite{Komornik_2012}) we introduce the sets
\begin{equation*}
\vv_q:=\set{x\in I_q: x\textrm{ has at most one doubly infinite }q\textrm{-expansion}},\quad q\in(1,2].
\end{equation*}
Here and in the sequel an expansion is called \emph{infinite} if it does not end  with $10^\f$ (it does not have a last one digit), and \emph{doubly infinite} if it does not end  with $01^\f$ or $10^\f$ (it has neither a last one digit, nor a last zero digit). 

{
\begin{remark}\label{r12}
If $q$ is not an integer, then each $x\in I_q$  \emph{has} a doubly infinite $q$-expansion, namely its quasi-greedy expansion.
See also \cite{Steiner}.
\end{remark}
}

We recall from \cite{DeVries_Komornik_2008} that $\vv_q$ is closed, and
\begin{equation*}
\uu_q\subseteq\overline{\uu_q}\subseteq\vv_q
\end{equation*}
for all $q\in(1,2]$, where $\overline{\uu_q}$ denotes  the topological closure of $\uu_q$. 

Using these sets we add two new characterizations of $\bb_2$ to Theorem \ref{t11} (i):

\begin{theorem}\label{t13}
The following conditions are equivalent:
\begin{enumerate}[{\rm(i)}]
\item $q\in\bb_2$;
\item $1\in\uu_q-\uu_q$;
\item $1\in\overline{\uu_q}-\overline{\uu_q}$;
\item $1\in\vv_q-\vv_q$ and $q\ne  \varphi$.
\end{enumerate}
\end{theorem}

In order to state our next results we recall from \cite{Komornik_Loreti_2007} and \cite{Komornik_2012} the notation
\begin{equation*}
\vv:=\{q\in(1,2]: 1\in\vv_q\};
\end{equation*}
this is the set of bases $q\in(1,2]$ in which $1$ has a unique doubly infinite $q$-expansion.

{
\begin{remark}\label{r14}
The number $x=1$ always has a doubly infinite $q$-expansion, namely its quasi-greedy expansion, even in integer bases.
\end{remark}
}

We recall from \cite{Komornik_Loreti_2007} that $\vv$ is closed, and that 
\begin{equation*}
\uu\subset\overline{\uu}\subset\vv.
\end{equation*}
The smallest element of $\vv$ is the Golden Ratio $\varphi$, while the smallest element of $\uu$ (and 
$\overline{\uu}$) is $q_{KL}$.
Furthermore, the difference set $\overline{\uu}\setminus\uu$ is countably infinite and dense in $\overline{\uu}$, and the difference set $\vv\setminus\overline{\uu}$ is discrete, countably infinite and dense in $\vv$.  

We recall from \cite{DeVries_Komornik_2008,Vries-Komornik-Loreti-2016} that $(1,2]\setminus\overline{\uu}=(1,2)\setminus\overline{\uu}$ is the disjoint union of its connected components $(q_0, q_0^*)$, where $q_0$ runs over $\set{1}\cup(\overline{\uu}\setminus\uu)$ and $q_0^*$ runs over a proper subset $\uu^*$ of $\uu$. 
Since the Komornik--Loreti constant $q_{KL}$ is the smallest element of $\overline{\uu}$, the first connected component is $(1, q_{KL})$.

We recall that each left endpoint $q_0$ is an algebraic integer, and each right endpoint $q_0^*$, called a \emph{de Vries--Komornik number}, is a transcendental number; see Kong and Li \cite{Kong_Li_2015}.

We also recall that for each component, $\vv\cap (q_0, q_0^*)$ is formed by an increasing sequence $q_1<q_2<\cdots$ of algebraic integers, converging to $q_0^*$. 
For example, the first two elements of $\vv$ in the first connected component $(1, q_{KL})$ are the Golden Ratio $\varphi$ and the second smallest element $q_f$ of $\bb_2$; see also Example \ref{e26} below.

Now we state our basic results on $\bb_2$:

\begin{theorem}\label{t15}\mbox{}
\begin{enumerate}[\upshape (i)]
\item The set $\bb_2$ is compact.
\item $\vv\setminus\set{\varphi}\subset\bb_2$.
\item Each element of $\vv\setminus\set{\varphi}$ is an accumulation point of $\bb_2$. 
Hence $\bb_2$ has infinitely many accumulation points in each connected component $(q_0, q_0^*)$ of $(1,2)\setminus\overline{\uu}$.
\item The smallest accumulation point of $\bb_2$ is its second smallest element $q_f$.
\item $\bb_2\cap(1,q_{KL})$ contains only algebraic integers, and hence it is countable.
\item $\bb_2$ has infinitely many isolated points in $(1, q_{KL})$, and  they are dense in $\bb_2\cap (1, q_{KL})$.
\end{enumerate}
\end{theorem}

\begin{remarks}\label{r16}\mbox{}

\begin{itemize}
\item Theorem \ref{t15} (i) answers positively Sidorov's question Q1. 
\item  Theorem \ref{t15} (ii) improves Theorem \ref{t11} (ii) because  $\uu$ (and even $\uuu$) is a proper subset of $\vv\setminus\set{\varphi}$.  
\item Since $(1,q_{KL})$ is a connected component of $(1,2)\setminus\overline{\uu}$, Theorem \ref{t15} (iii) answers negatively Sidorov's question Q2.
\item Theorem \ref{t15} (iv) answers partially a question of Baker and Sidorov \cite{Baker_Sidorov_2014} about the smallest accumulation point of $\bb_k$ for $k\ge 2$.
We recall that the smallest accumulation point $q_f$ of  $\bb_2$ is the smallest element of $\bb_k$ for all $k\ge 3$.
\item   Theorem \ref{t15} (v) strengthens a result of Sidorov \cite{Sidorov_2009} stating that $\bb_2\cap(1, q_{KL})$ contains only algebraic numbers.
\end{itemize}
\end{remarks}

In the following theorem we show that $\bb_2$ contains infinitely many accumulation points of all finite orders in $(1, q_{KL})$. 
For this we introduce the \emph{derived sets}  $\bb_2^{(0)}, \bb_2^{(1)},\ldots $ by induction, setting 
$\bb_2^{(0)}:=\bb_2$, and then 
\begin{equation*}
\bb_2^{(j+1)}=\{q\in(1,2]: q\textrm{ is an accumulation point of }\bb_2^{(j)}\}
\end{equation*}
for $j=0,1,\ldots .$

All these sets are compact by Theorem \ref{t15} (i) and a general property of derived sets. 
Since 
\begin{equation*}
\bb_2=\bb_2^{(0)}\supseteq\bb_2^{(1)}\supseteq\bb_2^{(2)}\supseteq\cdots,
\end{equation*}
the derived set
\begin{equation*}
\bb_2^{(\f)}:=\cap_{j=0}^\f\bb_2^{(j)}
\end{equation*}
of infinite order is also well defined, non-empty and compact.

\begin{theorem}\label{t17}\mbox{}
\begin{enumerate}[\upshape (i)]
\item $\overline{\uu}\subset\bb_2^{(\f)}$.
\item $\vv\setminus\set{\varphi}\subset\bb_2^{(2)}$.
\item $\min\bb_2^{(1)}=\min\bb_2^{(2)}=q_f$. 
\item All sets $\bb_2, \bb_2^{(1)}, \bb_2^{(2)},\ldots$ have   infinitely many   accumulation points in each connected component $(q_0, q_0^*)$ of $(1,2]\setminus\overline{\uu}$.
\item If $q_1<q_2<\cdots$ are the elements of $\vv\cap(1, q_{KL})$, then 
\begin{equation*}
q_{j+1}\le\min \bb_2^{(2j)}< q_{2j+1}\qtq{for all}j\ge 0,
\end{equation*}
and hence 
\begin{equation*}
\min \bb_2^{(j)}\nearrow\min\bb_2^{(\f)}=q_{KL}\qtq{as}j\ra\f.
\end{equation*}
\item For each $j=0,1,\ldots,$ $\bb_2^{(j)}\cap (1, q_{KL})$ has infinitely many isolated points, and  they are dense in $\bb_2^{(j)}\cap (1, q_{KL})$.
\end{enumerate}
\end{theorem}

\begin{remark}\label{r18}
Theorem \ref{t17} (iv) provides a negative answer to the question Q2 even if we replace $\bb_2$ by $\bb_2^{(j)}$.
\end{remark}
  
Our last result is related to the local dimension of $\bb_2$.

\begin{theorem}\label{t19}
For any $q\in\bb_2$ we have
\begin{equation*}
\lim_{\de\ra 0}\dim_H\big(\bb_2\cap(q-\de, q+\de)\big)\le 2 \dim_H\uu_q.
\end{equation*}
\end{theorem}

\begin{remarks}\label{r110}\mbox{}
 
\begin{itemize}
\item We recall from \cite{Komornik_Kong_Li_2015_1} that the dimension function  $q\mapsto \dim_H\uu_q$ is continuous, and vanishes at the Komornik--Loreti constant $q_{KL}$ (see also Lemma \ref{l24} below). 
Since $q_{KL}\in\uu\subset\bb_2$, by Theorem \ref{t15} (v) and Theorem \ref{t19} there exists a $\de>0$ such that 
\begin{equation*}
\dim_H(\bb_2\cap(q_{KL}, q_{KL}+\de))=\dim_H(\bb_2\cap(q_{KL}-\de, q_{KL}+\de))<1.
\end{equation*}
This answers affirmatively Sidorov's question Q3.
\item A related result, recently obtained  in \cite{Kong_Li_Lv_Vries2016}, states that
\begin{equation*}
\dim_H\big(\bb_2\cap(q_{KL},q_{KL}+\de)\big)>0\quad\textrm{for all}\quad \de>0.
\end{equation*}
\end{itemize}
\end{remarks}

The rest of the paper is arranged in the following way. 
In Section \ref{s2} we recall some results from unique expansions. 
Based on the properties of unique expansions we describe the set $\bb_2$ in Section \ref{s3}. 
In Section \ref{s4} we prove Theorem \ref{t13},  and the first parts of Theorems \ref{t15} and \ref{t17}.
Section \ref{s5} is devoted to the detailed description of unique expansions.
This important tool is extensively used in the rest of the paper.
The remaining parts of Theorems \ref{t15} and \ref{t17} are proved in Sections \ref{s6} and \ref{s7}, and the proof of Theorem \ref{t19} is given in Section \ref{s8}.
In the final section \ref{s9} we formulate some open questions related to the results of this paper.

\section{Unique expansions}\label{s2}

In this section we  recall several  results on unique expansions that will be used to prove our main theorems.  
Let $\set{0,1}^\f$ be the set of sequences $(d_i)=d_1d_2\cdots $ with elements $d_i\in\set{0,1}$. 
Let $\si$ be the left shift on $\set{0,1}^\f$ defined by $\si((d_i)):=(d_{i+1})$. 
Then $(\set{0,1}^\f, \si)$ is a full shift. 
Accordingly, let $\set{0,1}^*$ denote the set of all finite strings of zeros and ones, called \emph{words}, together with the empty word denoted by $\epsilon$. 
For a word $\c=c_1\cdots  c_m\in\set{0,1}^*$ we denote by $\c^k=(c_1\cdots  c_m)^k$ the $k$-fold  concatenation of $\c$ with itself, and by $\c^\f=(c_1\cdots  c_m)^\f$ the periodic sequence with period block $\c$. 
For a word $\c=c_1\cdots  c_m$ we set
\begin{equation*}
\c^+:=c_1\cdots  c_{m-1}(c_m+1)
\end{equation*}
if $c_m=0$, and 
\begin{equation*}
\c^-:=c_1\cdots  c_{m-1}(c_m-1)
\end{equation*}
if $c_m=1$.
The \emph{reflection} of a word $\c=c_1\cdots  c_m$ is defined by the formula
\begin{equation*}
\overline{\c}:=(1-c_1)\cdots (1-c_m),
\end{equation*}
and the reflection of a sequence $(d_i)\in\set{0,1}^\f$ is defined by
\begin{equation*}
\overline{(d_i)}=(1-d_1)(1-d_2)\cdots.
\end{equation*}

In this paper we use lexicographical orderings $<$, $\le$, $>$ and $\ge$ between sequences and words. 
Given two sequences $(c_i), (d_i)\in\set{0,1}^\f$ we write $(c_i)<(d_i)$ or $(d_i)>(c_i)$ if there exists an $n\in\NN$ such that $c_1\cdots  c_{n-1}=d_1\cdots  d_{n-1}$ and $c_n<d_n$. 
Furthermore, we write $(c_i)\le (d_i)$ or $(d_i)\ge (c_i)$ if $(c_i)<(d_i)$ or $(c_i)=(d_i)$. 
Finally, for two words $\mathbf{u}, \mathbf{v}\in\set{0,1}^*$ we write $\mathbf{u}<\mathbf{v}$ or $\mathbf{v}>\mathbf{u}$ if $\mathbf{u}0^\f<\mathbf{v}0^\f$.

Now we recall some notations and results from unique expansions.   
Given a base $q\in(1,2]$ we denote by $\beta(q)=(\beta_i(q))$ the lexicographically largest (called \emph{greedy}) expansion of $1$ in base $q$. Accordingly, we denote by
\begin{equation*}
\al(q)=\al_1(q)\al_2(q)\cdots 
\end{equation*}
the lexicographically largest infinite (called \emph{quasi-greedy}) expansion of $1$ in base $q$. 
Here an expansion is \emph{infinite} if it has infinitely many one digits. 
If $\beta(q)=\beta_1(q)\cdots \beta_n(q)0^\f$ with $\beta_n(q)=1$, then $\al(q)$ is periodic and $\al(q)=(\beta_1(q)\cdots  \beta_n(q)^-)^\f$. 

The following Parry type property (see \cite{Parry_1960}) of $\al(q)$ was given in \cite{Baiocchi_Komornik_2007}:

\begin{lemma}\label{l21}
The map $q\mapsto \al(q)$ is a strictly increasing bijection from $(1,2]$ onto the set of all infinite sequences $(a_i)\in\set{0,1}^\f$ satisfying
\begin{equation*}
a_{n+1}a_{n+2}\cdots \le a_1a_2\cdots \qtq{whenever} a_n=0.
\end{equation*}
\end{lemma}

We recall that $\uu_q$ is the set of $x\in[0, 1/(q-1)]$ having a unique $q$-expansion. 
We denote by $\uu_q'$ the set of expansions of all  $x\in \uu_q$. 
The following lexicographical characterization is a simple variant of another one given by Erd\H{o}s et al.~\cite{Erdos_Joo_Komornik_1990}.

\begin{lemma}\label{l22}
Let $q\in(1,2]$. 
Then $(c_i)\in\uu_q'$ if and only if
\begin{align*}
c_{n+1}c_{n+2}\cdots  <\al(q)&\qtq{whenever} c_n=0,\\
\overline{c_{n+1}c_{n+2}\cdots }<\al(q)&\qtq{whenever}c_n=1.
\end{align*}
\end{lemma}

Lemmas \ref{l21} and \ref{l22} imply that the set-valued map $q\mapsto \uu_q'$ is increasing, i.e.,  $\uu_p'\subseteq\uu_q'$ if $p<q$.

We recall that $\vv_q$ is the set of $x\in I_q$ having at most one doubly infinite $q$-expansion. 
The following lexicographical characterization of $\vv_q$ was given in \cite{DeVries_Komornik_2008}.

\begin{lemma}\label{l23}
Let $q\in(1,2]$. 
Then $x\in\vv_q$ if and only if $x$ has a $q$-expansion $(c_i)$ satisfying 
\begin{align*}
c_{n+1}c_{n+2}\cdots  \le \al(q)&\qtq{whenever} c_n=0,\\
\overline{c_{n+1}c_{n+2}\cdots }\le \al(q)&\qtq{whenever}c_n=1.
\end{align*}
\end{lemma}
Lemma  \ref{l23} implies that $\vv_q$ is closed for any $q\in(1,2]$. 
Then using Lemma \ref{l22} we conclude that 
\begin{equation*}
\uu_q\subseteq\overline{\uu_q}\subseteq\vv_q\quad\textrm{for every}\quad q\in(1,2].
\end{equation*}
Moreover, the difference set $\vv_q\setminus\uu_q$ is at most countable. It might happen that $\vv_q=\uu_q$ (see Lemma \ref{l28} below). 

For any $n\ge 1$ let $\L_n(\uu_q')$ be the set of   length $n$ subwords of sequences in $\uu_q'$, and let $\# A$ denote the cardinality of a set $A$. Then the \emph{topological entropy} of $\uu_q'$ is defined by
\begin{equation*}
h_{top}(\uu_q')=\lim_{n\ra\f}\frac{\log\# \L_n(\uu_q')}{n}=\inf_{n\ge 1}\frac{\log\# \L_n(\uu_q')}{n}.
\end{equation*}
The above limit exists for each $q>1$ by \cite[Lemma 2.1]{Komornik_Kong_Li_2015_1}. 
Moreover, the following lemma was proved in  \cite{Komornik_Kong_Li_2015_1}:
 
\begin{lemma}\label{l24}
The Hausdorff dimension of $\uu_q$ is given by
\begin{equation*}
\dim_H\uu_q=\frac{h_{top}(\uu_q')}{\log q}
\end{equation*}
for every $q\in(1,2]$. 
Furthermore, the dimension function $D: q\mapsto \dim_H\uu_q$ has a Devil's staircase behavior:
\begin{itemize}
\item $D$ is continuous, and has bounded variation in $(1,2]$;
\item $D'<0$ almost everywhere in $[q_{KL},2]$;
\item $D(q)=0$ for all $1<q\le q_{KL}$, and $D(q)>0$ for all $q>q_{KL}$.
\end{itemize}
\end{lemma}

We recall that $\uu$ is the set of bases $q\in(1,2]$ in which  $1$ has a unique $q$-expansion, and $\vv$ is the set of bases $q\in(1,2]$ in which $1$ has a unique  doubly infinite $q$-expansion. 
 
The following lexicographical characterizations of $\uu$, its closure $\overline{\uu}$ and $\vv$ are due to Komornik and Loreti \cite{Komornik_Loreti_2007} (see also,  \cite{Vries-Komornik-Loreti-2016}).
 
\begin{lemma}\label{l25}\mbox{}
\begin{enumerate}[\rm(i)]
\item $q\in\uu\setminus\set{2}$ if and only if $\al(q)=(\al_i(q))$ satisfies
\begin{equation*}
\overline{\al(q)}<\al_{n+1}(q)\al_{n+2}(q)\cdots <\al(q)\qtq{for all} n\ge 1.
\end{equation*}
\item $q\in\overline{\uu}$ if and only if $\al(q)=(\al_i(q))$ satisfies
\begin{equation*}
\overline{\al(q)}<\al_{n+1}(q)\al_{n+2}(q)\cdots \le\al(q)\qtq{for all} n\ge 1.
\end{equation*}
\item $q\in\vv$ if and only if $\al(q)=(\al_i(q))$ satisfies
\begin{equation}\label{21}
\overline{\al(q)}\le \al_{n+1}(q)\al_{n+2}(q)\cdots \le\al(q)\qtq{for all} n\ge 1.
\end{equation}
\end{enumerate}
\end{lemma}

\begin{example}\label{e26}
For the first two elements of $\vv$ we have $\al(\varphi)=(10)^\f$ and $\al(q_f)=(1100)^\f$.
\end{example}

By Lemma \ref{l25} it is clear that $\uu\subset\overline{\uu}\subset\vv$. Furthermore, $\vv$ is closed.  
In order to prove our main results we will also need 
the following topological properties of $\uu, \overline{\uu}$ and $\vv$ (see~\cite{Komornik_Loreti_2007}):
 
\begin{lemma}\label{l27}\mbox{}
\begin{enumerate}[{\rm(i)}]
\item $\overline{\uu}\setminus\uu$ is a countable dense subset of $\overline{\uu}$, and $\vv\setminus\overline{\uu}$ is a discrete dense subset of $\vv$. 
\item For each $q\in\overline{\uu}\setminus\uu$ there exists a sequence $(p_n)$ in $\uu$   that $p_n\nearrow q$ as $n\ra\f$.
\item For each $q\in\overline{\uu}\setminus\uu$ the quasi-greedy expansion $\al(q)$ is periodic. 
\item For each $q\in\vv\setminus(\overline{\uu}\cup\set{\varphi})$ there exists a word $a_1\cdots  a_m$ with $m\ge 2$ such that 
\begin{equation*}
\al(q)=(a_1\cdots  a_m^+\,\overline{a_1\cdots  a_m^+})^\f,
\end{equation*}
where $(a_1\cdots  a_m)^\f$ satisfies  \eqref{21}.
\end{enumerate}
\end{lemma}

Finally,  we recall from \cite{DeVries_Komornik_2008} the  following relation between  the sets $\overline{\uu}, \vv, \uu_q, \overline{\uu_q}$ and $\vv_q$.
 
\begin{lemma}\label{l28}\mbox{}
\begin{enumerate}[{\rm(i)}]
\item $\uu_q$ is closed if and only if $q\in(1,2]\setminus\overline{\uu}$.
\item $\uu_q=\overline{\uu_q}=\vv_q$ if and only if $q\in(1,2]\setminus\vv$.
\item  Let $(r_1, r_2)$ be a connected component  of $(1,2]\setminus\vv$.     Then $\uu_q'=\uu_{r_2}'$ {for all}  $q\in(r_1, r_2]$. 
\end{enumerate}
\end{lemma}

\section{Description  of $\bb_2$}\label{s3}

In this section we describe the set $\bb_2$ by using unique expansions. 
The following result is essentially equivalent to Theorem \ref{t11} (i):  

\begin{lemma}\label{l31}
A number $q\in (1,2]$ belongs to $\bb_2$ if and only if there exist two sequences $(c_i),(d_i)\in\uu_q'$ satisfying the equality
\begin{equation}\label{31}
(1(c_i))_q=(0(d_i))_q.
\end{equation}
\end{lemma}

\begin{proof}
The sufficiency of the condition \eqref{31} is obvious.
Conversely, assume that a real number $x$ has exactly two $q$-expansions, say $(a_i)$ and $(b_i)$.
Then, assuming by symmetry that $(a_i)>(b_i)$ lexicographically, there exists a first index $n$ such that $a_n\ne b_n$.
Then $a_n=1$, $b_n=0$, $(a_{n+i}),(b_{n+i})\in \uu_q'$, and \eqref{31} is satisfied with $(c_i)=(a_{n+i})$ and $(d_i)=(b_{n+i})$.\qed
\end{proof}

For each $q\in(1,2]$ we define
\begin{equation*}
A_q':=\set{(c_i)\in\uu_q': c_1= 0}.
\end{equation*}
Each sequence $(c_i)\in A_q'$ satisfies
\begin{equation*}
c_{n+1}c_{n+2}\cdots <\alpha(q)\qtq{for any} n\ge 0,
\end{equation*}
and
\begin{equation*}
\uu_q'=\bigcup_{\c\in A_q'}\set{\c,\overline{\c}}
\end{equation*}
by Lemma \ref{l22}.

In the following  improvement of Lemma \ref{l31} we use only sequences from $A_q'$:

\begin{lemma}\label{l32}
A number  $q\in(1,2]$ belongs to $\bb_2$ if and only if $q$ is a zero of the function
\begin{equation*}
f_{\c,\d}(q):=(1\c)_q+(1\d)_q-(1^\f)_q
\end{equation*}
for some $\c,\d\in A_q'$.
\end{lemma}

\begin{proof}
Let $q\in(1,2]$. 
Since $\uu_q'=\bigcup_{\c\in A_q'}\set{\c, \overline{\c}}$,  by Lemma \ref{l31} we have $q\in \bb_2$ if and only if $q$ satisfies one of the following equations for some $\c, \d\in A_q'$:
\begin{equation*}
(1\c)_q=(0\d)_q,\quad (1\c)_q=(0\overline{\d})_q,\quad (1\overline{\c})_q=(0\d)_q\quad\textrm{and}\quad (1\overline{\c})_q=(0\overline{\d})_q.
\end{equation*}
It follows by reflection that the first and the forth equalities are equivalent. 
It remains to prove that the first and the third equalities never hold.

Since $\d\in A_q'$, by Lemma \ref{l22} we have
\begin{equation*}
(0\d)_q<(00\al(q))_q<(10^\f)_q\le (1\c)_q.
\end{equation*}
Similarly, $(0\d)_q<(1\overline{\c})_q$. 
This implies that $q\in \bb_2$ if and only if $q$ is  a zero of the function
\begin{equation*}
f_{\c, \d}(q)=(1\c)_q-(0\overline{ \d})_q=(1\c)_q+(1\d)_q-(1^\f)_q. 
\end{equation*}\qed
\end{proof}

In view of Lemma \ref{l32} we are led to investigate the functions $f_{\c,\d}$ for $\c,\d\in A_q'$.

\begin{lemma}\label{l33} 
Let $q\in(1,2]$ and $\c,\d\in A_q'$.
\begin{enumerate}[\upshape (i)]
\item The function  $f_{\c,\d}$ is continuous and symmetric: $f_{\c,\d}=f_{\d,\c}$.
\item If $\tilde{\c}\in A_q'$ and $\tilde\c>\c$, then $f_{\tilde\c,\d}>f_{\c,\d}$.
\item If $\tilde{\d}\in A_q'$ and $\tilde\d>\d$, then $f_{\c,\tilde\d}>f_{\c,\d}$.
\item  $f_{\c,\d}(2)\ge 0$.
\end{enumerate}
\end{lemma}

\begin{proof}
(i) follows from the definition of $f_{\c,\d}$, (ii) follows from the definition of $f_{\c,\d}$ and the property of unique expansions that
\begin{equation*}
(1\tilde{\c})_q>(1\c)_q
\end{equation*}
for any $\tilde{\c}, \c\in A_q'$ with $\tilde{\c}>\c$. 
(iii) follows from (i) and (ii). Finally, (iv) follows from (ii), (iii) and the equality
\begin{equation*}
f_{\c,\d}(2)=(10^\f)_2+(10^\f)_2-(1^\f)_2=0
\end{equation*}
when $\c=\d=0^\f$ (the lexicographically smallest element of $A_q'$).\qed
\end{proof}
 
Lemma \ref{l32} states that each $q\in\bb_2$ is a zero  of the function $f_{\c,\d}$ for some $\c,\d\in A_q'$. We prove in our next lemma that  no function $f_{\c, \d}$ with $\c,\d\in A_q'$ can provide more than one element of $\bb_2\cap[q, 2]$.

Since $\bb_2\cap(1, q_f)$ contains only one element by Theorem \ref{t11} (iv): the positive zero $q_s\approx 1.71064$ of $q^4-2 q^2-q-1$, henceforth we restrict our attention to the set $\bb_2\cap[q_f, 2]$.

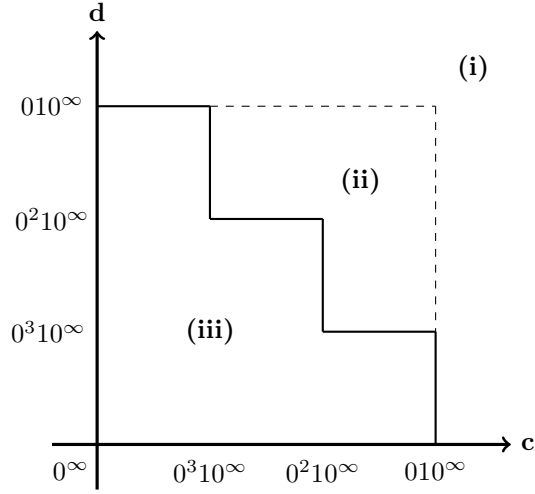
\begin{figure}[h!]
\begin{center}
\begin{tikzpicture}[
    scale=5,
    axis/.style={very thick, ->},
    important line/.style={thick},
    dashed line/.style={dashed, thin},
    pile/.style={thick, ->, >=stealth', shorten <=2pt, shorten
    >=2pt},
    every node/.style={color=black}
    ]
    \draw[axis] (-0.12,0)  -- (1.1,0) node(xline)[right]
        {$\mathbf{c}$};
    \draw[axis] (0,-0.12) -- (0,1.1) node(yline)[above] {$\mathbf{d}$};
     
    \node[] at (-0.07,-0.07){$0^\infty$};
    \draw[important line]  (0,0.9)--(0.3,0.9);
    \draw[dashed line] (0.3,0.9)--(0.9,0.9);
      \node[] at (-0.12,0.9){$010^\infty$};
         \node[] at (-0.12,0.6){$0^210^\infty$};
            \node[] at (-0.13,0.3){$0^310^\infty$};
     \draw[dashed line] (0.9,0.9)--(0.9, 0.3);       
    \draw[important line] (0.9, 0.3)  -- (.9,0);
   \node[] at (0.9, -0.07){$010^\infty$};
   
     \node[] at (0.6, -0.07){$0^210^\infty$};
     \node[] at (0.3, -0.07){$0^310^\infty$};
     
   \draw[important line](0.3,0.9)--(0.3,0.6);
   \draw[important line](0.3,0.6)--(0.6,0.6);
   \draw[important line](0.6,0.6)--(0.6,0.3);
   \draw[important line](0.6,0.3)--(0.9,0.3);
   
     
     \node[] at (0.3,0.3){\textbf{(iii)}};
      \node[] at (0.7,0.7){\textbf{(ii)}};
       \node[] at (1.0,1.0){\textbf{(i)}};
\end{tikzpicture} 
\end{center}
\caption{Different cases of Lemma \ref{l34}.}\label{f1}
\end{figure}

\begin{lemma}\label{l34}
Let $q\in[q_f, 2]$ and $\c, \d\in A_q'$.
\begin{enumerate}[{\rm(i)}]
\item  If $\c\ge 010^\f$ or $\d\ge 010^\f$, then $f_{\c, \d}(q)>0$.
\item If $\c\ge 0^310^\f$ and $\d\ge 0^210^\f$, or $\c\ge 0^210^\f$ and $\d\ge 0^310^\f$, then $f_{\c,\d}(q)>0$.
\item In all other cases the function $f_{\c,\d}$ is strictly increasing in $[q, 2]$.
\end{enumerate}
\end{lemma}

\begin{proof}
(i) Let $q\in[q_f, 2]$. 
Applying Lemma \ref{l22} we see that $0^\f$ is the lexicographically  smallest element of $A_q'$ and $010^\f\notin A_q'$. 
In view of Lemma \ref{l33} it suffices to prove that
\begin{equation*}
f_{\c,\d}(q)\ge 0\qtq{when} \c=010^\f\textrm{ and }\d=0^\f.
\end{equation*}
This follows from the following computation:
\begin{align*}
f_{\c,\d}(q) 
&=(1010^\f)_q+(10^\f)_q-(1^\f) \\
& =\frac{2}{q}+\frac{1}{q^3}-\frac{1}{q-1} \\
& =\frac{q^3-2q^2+q-1}{q^3(q-1)}\;\ge 0.
\end{align*}
The last inequality holds because
\begin{equation*}
q^3-2q^2+q-1\ge q_f^3-2q_f^2+q_f-1=0.
\end{equation*}
\medskip 

(ii) We recall that $\varphi_3\approx 1.83929$ is the Tribonacci number,  the positive zero of   $q^3-q^2-q-1$. 
We distinguish the following two cases: (a) $q\in[q_f,\varphi_3]$ and (b) $q\in[\varphi_3, 2]$.

Case (a): $q\in[q_f,\varphi_3]$. 
Since $\al(\varphi_3)=(110)^\f$, by Lemma \ref{l21}  we have $\al(q)\le (110)^\f$. 
Then by Lemma \ref{l22} it follows that if $\c\in A_q'$ with $\c\ge 0^k10^\f$ for some $k\ge 0$, then $\c\ge 0^k(100)^\f$. 
In view of Lemma \ref{l33} it suffices to prove
\begin{equation*}
f_{\c,\d}(q)>0\qtq{when} \c=0^3(100)^\f\textrm{ and } \d=0^2(100)^\f.
\end{equation*}
This follows from the following observation:
\begin{align*}
f_{\c,\d}(q) 
&=(10^3(100)^\f)_q+(10^2(100)^\f)_q-(1^\f)_q \\
& =\frac{2}{q}+\frac{1}{q^2(q^3-1)}+\frac{1}{q(q^3-1)}-\frac{1}{q-1} \\
&=\frac{q^4-q^3-q^2-q+1}{q^2(q^3-1)}\;>0.
\end{align*}
Here the last inequality follows by using the inequalities
\begin{equation*}
q^4-q^3-q^2-q+1\ge q_f^4-q_f^3-q_f^2-q_f+1=2-q_f>0.
\end{equation*}

Case (b): $q\in[\varphi_3, 2]$. 
By Lemma \ref{l33} it suffices to prove that
\begin{equation*}
f_{\c,\d}(q)>0\qtq{when}\c=0^310^\f\textrm{ and }\d=0^210^\f.
\end{equation*}
This follows from the following computation:
\begin{align*}
f_{\c,\d}(q)
&=(10^310^\f)_q+(10^210^\f)_q-(1^\f)_q\\
&=\frac{2}{q}+\frac{1}{q^4}+\frac{1}{q^5}-\frac{1}{q-1}=\frac{q^5-2q^4+q^2-1}{q^5(q-1)}\;>0.
\end{align*}
The last inequality follows by using $q\ge \varphi_3$ and $\varphi_3^3=\varphi_3^2+\varphi_3+1$:
\begin{align*}
q^5-2q^4+q^2-1 
&\ge \varphi_3^5-2\varphi_3^4+\varphi_3^2-1 \\
& =-\varphi_3^4+\varphi_3^3+2\varphi_3^2-1=\varphi_3^2-\varphi_3-1>0.
\end{align*}
\medskip 

(iii) Let $q\in[q_f, 2]$ and $p_1, p_2\in[q,2]$ with $p_1<p_2$. 
We will prove the inequality $f_{\c,\d}(p_1)<f_{\c,\d}(p_2)$ by distinguishing two cases again (see Figure \ref{f1}).

\emph{First case: $\c<0^310^\f$ and $\d<010^\f$, or $\c<010^\f$ and $\d<0^310^\f$.}
Observe that for any sequence $(a_i)\in\set{0, 1}^\f$ and for any positive integer $n$ the elementary inequality
\begin{equation}\label{32}
0\le (a_1a_2\cdots)_{p_1}-(a_1a_2\cdots)_{p_2}\le (a_1\cdots a_n1^{\infty})_{p_1}-(a_1\cdots a_n1^{\infty})_{p_2}
\end{equation}
holds. 
By Lemma \ref{l33} it suffices to prove that
\begin{equation*}
f_{\c,\d}(p_1)<f_{\c,\d}(p_2)\qtq{when} \c\le 0^41^\f\textrm{ and }\d\le 0^21^\f.
\end{equation*}
Using \eqref{32}  with $n=5$ and $n=3$ we get
\begin{align*}
f_{\c,\d}({p_2}) -f_{\c,\d}({p_1})
&=[(1\c)_{p_2}-(1\c)_{p_1}]+[(1\d)_{p_2}-(1\d)_{p_1}]-[(1^\f)_{p_2} - (1^\f)_{p_1}]\\
&\ge[(10^41^\f)_{p_2}-(10^41^\f)_{p_1}]+[(10^21^\f)_{p_2}-(10^21^\f)_{p_1}]\\
&\quad -[(1^\f)_{p_2}-(1^\f)_{p_1}]\\
&=[(10^41^\f)_{p_2}+(10^21^\f)_{p_2}-(1^\f)_{p_2}]-[(10^41^\f)_{p_1}+(10^21^\f)_{p_1}-(1^\f)_{p_1}].
\end{align*}
The required inequality $f_{\c,\d}(p_2)>f_{\c,\d}(p_1)$ follows by observing that the function
\begin{equation*}
g_1(x):=(10^41^\f)_x+(10^21^\f)_x-(1^\f)_x=\frac{x^5-2x^4+x^2+1}{x^5(x-1)}
\end{equation*}
is strictly  increasing in $[q_f, 2]$.

\emph{Second case: $\c<0^210^\f$ and $\d<0^210^\f$.}
Now using \eqref{32} with $n=4$ we get, similarly to the first case, the inequality
\begin{equation*}
f_{\c,\d}({p_2})-f_{\c,\d}({p_1})\ge[2(10^31^\f)_{p_2}-(1^\f)_{p_2}]-[2(10^31^\f)_{p_1}-(1^\f)_{p_1}].
\end{equation*}
Note that the function
\begin{equation*}
g_2(x):=2(10^31^\f)_x-(1^\f)_x=\frac{x^4-2x^3+2}{x^4(x-1)}
\end{equation*}
is strictly  increasing in $[q_f, 2]$. 
Therefore we conclude that $f_{\c,\d}(p_2)>f_{\c,\d}(p_1)$. \qed
\end{proof}

Let $q\in[q_f, 2]$ and set
\begin{equation*}
\Omega_q':=\set{(\c,\d)\in A_q'\times A_q':  f_{\c, \d}(q)\le 0}.
\end{equation*}
If $(\c, \d)\in \Omega_q'$, then $f_{\c,\d}$ has a unique root in $[q,2]$ by Lemmas \ref{l33} (iv) and  \ref{l34}, and this root, denoted by $q_{\c,\d}$, belongs to $\bb_2$ by Lemma \ref{l32}.
We have
\begin{equation*}
\bb_2\cap [q, 2]=\bigcup_{p\in[q, 2]}\set{q_{\c,\d}\ :\ (\c, \d)\in \Omega_p'}.
\end{equation*}

An element of $\bb_2\cap [q_f, 2]$ may have multiple representations: see Remark \ref{r42} below.

\begin{lemma}\label{l35}
Let $q\in[q_f,2]$.
\begin{enumerate}[{\rm(i)}]
\item If $(\c,\d), (\tilde\c,\d)\in\Omega_q'$ and   $ \tilde\c>\c$, then  $q_{\tilde\c,\d}<q_{\c, \d}$;
\item If  $(\c,\d), (\c,\tilde\d)\in\Omega_q'$ and  $\tilde\d>\d$, then $q_{\c,\tilde\d}<q_{\c,\d}$.
\end{enumerate}
\end{lemma}

\begin{proof}
By symmetry we only prove (i).

Let $(\c,\d), (\tilde\c,\d) \in \Omega_q'$ with $ \tilde\c>\c$. 
Applying Lemma \ref{l33} and using the definitions of $q_{\c,\d}, q_{\tilde\c, \d}$ we have
\begin{align*}
f_{\tilde\c,\d}(q_{\tilde\c,\d})=0=f_{\c,\d}(q_{\c,\d})
&=(1\c)_{q_{\c,\d}}+(1\d)_{q_{\c,\d}}-(1^\f)_{q_{\c,\d}}\\
&<(1\tilde\c)_{q_{\c,\d}}+(1\d)_{q_{\c,\d}}-(1^\f)_{q_{\c,\d}}\\
&=f_{\tilde\c,\d}(q_{\c,\d}).
\end{align*}
Since $f_{\tilde\c,\d}$ is strictly increasing in $[q, 2]$ by Lemma \ref{l34}, we conclude that $q_{\tilde\c,\d}<q_{\c,\d}$. \qed
\end{proof}

\section{Proof of Theorem \ref{t13},  Theorem \ref{t15} (i)-(iv) and Theorem \ref{t17} (i)-(iii)}\label{s4}

We recall that $\vv$ is the set of bases $q\in(1,2]$ in which $1$ has a unique doubly infinite $q$-expansion;  equivalently, $\vv=\set{q\in(1,2]: 1\in\vv_q}$.  
In this section we show that $\vv\setminus\set{\varphi}\subset\bb_2$. 
Based on this observation we  give new characterizations of $\bb_2$ (Theorem \ref{t13}). 
We also show that $\bb_2$ is closed, $\vv\setminus\set{\varphi}\subset\bb_2^{(2)}$, and we conclude that the second smallest element $q_f$ of $\bb_2$ is also the smallest accumulation point of $\bb_2$ (Theorem \ref{t15}). 
   
We recall that $\uu$ is the set of univoque bases $q\in(1,2]$ in which $1$ has a unique $q$-expansion, i.e., $\uu=\set{q\in(1,2]: 1\in\uu_q}$. 
First we show that its topological closure $\overline{\uu}$ is a subset of $\bb_2$:

\begin{lemma}\label{l41}
We have $\overline{\uu}\subset\bb_2$, and even $\overline{\uu}\subset\bb_2^{(\f)}$.
\end{lemma}

\begin{proof}
Since $\overline{\uu}$ is a Cantor set and therefore $\overline{\uu}^{(\f)}=\overline{\uu}$, it suffices to prove that $\overline{\uu}\subset\bb_2$. 
Furthermore, since $\uu\subset\bb_2$ by Theorem \ref{t11}, it suffices to prove that $\overline{\uu}\setminus\uu\subset\bb_2$.

Take $q\in\overline{\uu}\setminus\uu$ arbitrarily. 
By Lemma \ref{l27} (iii)  there exists a word $a_1\cdots  a_m$ such that 
\begin{equation*}
\alpha(q)=(a_1\cdots a_m)^\f.
\end{equation*}
Suppose that $m$ is the smallest period of $\al(q)$. 
Then $m\ge 2$, and the greedy $q$-expansion of $1$ is given by $\beta(q)=a_1\cdots a_m^+0^\f$.
Observe  that $\si^n(\beta(q))<\beta(q)$ for all $n\ge 1$ (cf.~\cite{Parry_1960}). 
Then applying Lemma \ref{l25} (ii) it follows that
\begin{equation}\label{41}
\overline{a_1\cdots a_{m-i}}\le a_{i+1}\cdots a_m<a_{i+1}\cdots a_m^+\le  a_1\cdots a_{m-i}\qtq{for all} 0<i<m.
\end{equation}
 
Since $q\in\overline{\uu}\setminus\uu$, by Lemmas \ref{l21} and \ref{l27} (ii) there exists a $p\in\uu\cap(1, q)$ such that
\begin{equation}\label{42}
\al_1(p)\cdots \al_m(p)=\al_1(q)\cdots \al_m(q)=a_1\cdots a_m.
\end{equation}
In fact we could find infinitely many $p\in\uu\cap(1, q)$ satisfying \eqref{42}. 
Now we claim that
\begin{equation*}
\c:=\overline{a_1\cdots a_m^+}\alpha_1(p)\alpha_2(p)\cdots
\end{equation*}
belongs to $A_q'$, where $A_q'$ is the set of sequences in $\uu_q'$ with a prefix $0$.

First we prove that $\c\in\uu_q'$. 
Since $p\in\uu\cap(1,q)$, by Lemmas \ref{l21} and  \ref{l25} (i) it follows that 
\begin{equation*}
\overline{\al(q)}<\overline{\al(p)}\le \si^n(\al(p))\le \al(p)<\al(q)\qtq{for any} n\ge 0.
\end{equation*}
Then by Lemma \ref{l22}  it suffices to prove  
\begin{equation*}
\overline{\alpha(q)}< \sigma^i(\c)<\alpha(q)\qtq{for any}0< i<m.
\end{equation*}

By \eqref{41} it follows that for any $0<i<m$ we have
\begin{equation*}
\overline{a_{i+1}\cdots a_m^+}<\overline{a_{i+1}\cdots a_m}\le a_1\cdots a_{m-i},
\end{equation*}
which implies that $\si^i(\c)<\al(q)$ for any $0<i<m$.
Furthermore, by \eqref{41} we have
\begin{equation*}
\overline{a_{i+1}\cdots a_m^+}\ge\overline{a_1\cdots a_{m-i}}\qtq{and} a_1\cdots a_i\ge \overline{a_{m-i+1}\cdots a_m}.
\end{equation*}
This, together with  \eqref{42}, implies that
\begin{align*}
\si^i(\c)&=\overline{a_{i+1}\cdots a_m^+}\al_1(p)\al_2(p)\cdots \\
&=\overline{a_{i+1}\cdots a_m^+} a_1\cdots a_i \al_{i+1}(p)\al_{i+2}(p)\cdots\\
&\ge \overline{a_1\cdots a_m}\overline{\alpha(p)}\\
&>\overline{a_1\cdots a_m}\;\overline{\al(q)}=\;\overline{\al(q)};
\end{align*}
the last equality holds because $\al(q)=(a_1\cdots a_m)^\f$. 
Therefore, $\c\in\uu_q'$.
Since $\overline{a_1}=0$, we conclude that $\c\in A_q'$.

Now take $\d:=0^m\overline{\al(p)}$. 
Then by the above reasoning we also have $\d\in A_q'$.  
In view of Lemma \ref{l32}  the following calculation shows that $q\in\bb_2$:
\begin{align*}
f_{\c,\d}(q)&=(1\c)_q+(1\d)_q-(1^\f)_q\\
&=(1\overline{a_1\cdots a_m^+}\al(p))_q+(10^m\overline{\al(p)})_q-(1^\f)_q\\
&=(2\overline{a_1\cdots a_m^+}1^\f)_q-(1^\f)_q\\
&=(10^\f)_q-(0a_1\cdots a_m^+0^\f)_q=0;
\end{align*}
the last equality holds because $\beta(q)=a_1\cdots a_m^+ 0^\f$. \qed
\end{proof}

\begin{remark}\label{r42}
The proof of Lemma \ref{l41} shows that each $q\in\overline{\uu}\setminus\uu\subset\bb_2$ has infinitely many representations, i.e., $q=q_{\c,\d}$ for infinitely many pairs $(\c,\d)\in\Omega_q'$.
\end{remark}

The preceding lemma may be improved:

\begin{lemma}\label{l43}
$\vv\setminus\set{\varphi}\subset\bb_2$.
\end{lemma}

\begin{proof}
By Lemma \ref{l41} it remains to prove that each $q\in\vv\setminus(\overline{\uu}\cup\set{\varphi})$ belongs to $\bb_2$. 
Set
\begin{equation}\label{eq:v}
\vv'=\set{(a_i)\in\set{0,1}^\f: \overline{(a_i)}\le (a_{n+i})\le (a_i)\qtq{for all} n\ge 1}.
\end{equation}
By Lemmas \ref{l25} (iii) and \ref{l27} (iv) there exists a block $a_1\cdots  a_m$ with $m\ge 2$ such that 
\begin{equation*}
\al(q)=(a_1\cdots a_m^+\overline{a_1\cdots a_m^+})^{\infty}\in\vv'\qtq{and} (a_1\cdots a_m)^\f\in\vv'.
\end{equation*}
 
We claim that
\begin{equation*}
\c=\overline{a_1\cdots a_m^+}(a_1\cdots a_m)^\f
\end{equation*}
belongs to $A_q'$.

Since both $\al(q)=(a_1\cdots  a_m^+\overline{a_1\cdots  a_m^+})^\f$ and $(a_1\cdots a_m)^\f$ belong to $\vv'$, we have
\begin{align*}
&\overline{a_{i+1}\cdots a_m^+}<\overline{a_{i+1}\cdots a_m}\le  a_1\cdots a_{m-i}
\intertext{and}
&\overline{a_{i+1}\cdots a_m^+}a_1\cdots a_i\ge \overline{a_1\cdots a_{m-i}}\;\overline{a_{m-i+1}\cdots a_m}> \overline{a_{1}\cdots a_m^+}
\end{align*}
for all $0< i<m$.
This implies the relations
\begin{equation*}
(\overline{a_1\cdots a_m^+}a_1\cdots a_m^+)^\f< \si^i({\c})< (a_1\cdots a_m^+\overline{a_1\cdots a_m^+})^\f
\end{equation*}
for all $i>0$. 
Hence $\c\in\uu_q'$. 
Since $\overline{a_1}=0$, we conclude that $\c\in A_q'$.

Let  $\d:=0^{2m}(\overline{a_1\cdots a_m})^\f$, then $\d\in A_q'$. 
We complete the proof of $q\in\bb_2$ by showing that $f_{\c,\d}(q)=0$.
This follows from the following computation:
\begin{align*}
f_{\c,\d}(q)&=(1\overline{a_1\cdots a_m^+}(a_1\cdots a_m)^\f)_q+(10^{2m}(\overline{a_1\cdots a_m})^\f)_q-(1^\f)_q\\
&=(2\overline{a_1\cdots a_m^+}a_1\cdots a_m 1^\f)_q-(1^\f)_q\\
&=(10^\f)_q-(0a_1\cdots a_m^+\overline{a_1\cdots a_m}0^\f)_q=0;
\end{align*}
the last equality holds because $\beta(q)=a_1\cdots a_m^+\overline{a_1\cdots a_m}0^\f$. \qed
\end{proof}

We recall that $\uu_q$ is the set of $x\in I_q$ having a unique $q$-expansion, $\overline{\uu_q}$ is its topological closure, and $\vv_q$ is the set of $x\in I_q$ having at most one doubly infinite $q$-expansion. 
We also recall that $\uu_q\subseteq\overline{\uu_q}\subseteq\vv_q$ for all $q\in(1,2]$.

Lemma \ref{l43} allows us to give new characterizations of $\bb_2$:
 
\begin{proposition}\label{p44}
Set 
\begin{equation*}
\C_2:=\set{q\in(1,2]: 1\in\overline{\uu_q}-\overline{\uu_q}}\qtq{and}\D_2:=\set{q\in(1,2]: 1\in\vv_q-\vv_q}.
\end{equation*}
Then $\bb_2=\C_2=\D_2\setminus\set{\varphi}$.
\end{proposition}

\begin{proof} 
We already know from Theorem \ref{t11} (i) that 
\begin{equation*}
\bb_2=\set{q\in(1,2]: 1\in\uu_q-\uu_q}.
\end{equation*}
Since $\uu_q\subseteq\overline{\uu_q}\subseteq\vv_q$ for all $q\in(1,2]$, we have
\begin{equation}\label{43}
\bb_2\subseteq \C_2\subseteq\D_2.
\end{equation}

Since $\overline{\uu_q}=\uu_q$ for all $q\in(1,2]\setminus\overline{\uu}$ by Lemma \ref{l28} (i), we infer from Lemma \ref{l41} that
\begin{equation*}
\C_2\subseteq\bb_2\cup\overline{\uu}=\bb_2.
\end{equation*}
Combining this with \eqref{43} we conclude that $\C_2=\bb_2$.

Next, since $\vv_q=\uu_q$ for all $q\in(1,2]\setminus\vv$ by Lemma \ref{l28} (ii), using also Lemma \ref{l43} we obtain that
\begin{equation*}
\D_2\subseteq\bb_2\cup\vv=\bb_2\cup\set{\varphi}.
\end{equation*}

Observe that the condition $1\in\vv_q-\vv_q$ is satisfied for $q=\varphi$ because $0,1\in\vv_{\varphi}$. 
This implies that $\varphi\in\D_2$.
Combining this with \eqref{43} and $\varphi\notin\bb_2$ we conclude that $\bb_2=\D_2\setminus\set{\varphi}$. \qed
\end{proof}

\begin{lemma}\label{l45}
$\bb_2$ is compact.
\end{lemma}

\begin{proof}
Since $\bb_2$ is bounded, it suffices to prove that its complement is open.
Since by Theorem \ref{t11} (iv) that $\bb_2$ has a smallest point $q_s\approx 1.71064$, it suffices to show that $(q_s,2]\setminus\bb_2$ is a neighborhood of each of its points.

Fix $q\in (q_s,2]\setminus\bb_2$ arbitrarily. 
Since  $\varphi\approx 1.61803<q_s$, we have $q\notin\vv$ by Lemma \ref{l43}.
Consider the connected component $(r_1,r_2)$ of the open set $(1,2)\setminus\vv$ that contains $q$.

Since $q\notin\uuu$,  by Lemma \ref{l28} (i) the set  $\uu_q$ is compact, and hence $\uu_q-\uu_q$ is compact.
Since by Theorem \ref{t11} (i) that  $q\notin\bb_2$ is equivalent to $1\notin\uu_q-\uu_q$, hence $\dist(1,\uu_q-\uu_q)>0$.

By Lemma \ref{l28} (iii) we have $\uu_p'=\uu_q'$ for all $p\in (r_1,r_2)$.
Therefore the sets $\uu_p$ depend continuously on $p\in (r_1,r_2)$, and hence the inequality $\dist(1,\uu_p-\uu_p)>0$ remains valid in a small neighborhood $G$ of $q$.
Then $1\notin\uu_p-\uu_p$, i.e., $p\notin\bb_2$ for all $p\in G$. \qed
\end{proof}

Next we improve Lemma \ref{l43}: each $q\in\vv\setminus\set{\varphi}$ is even an accumulation point of $\bb_2$.
In fact, the following stronger result holds:

\begin{lemma}\label{l46}
$\vv\setminus\set{\varphi}\subset\bb_2^{(2)}$.
\end{lemma}

\begin{proof}
In view of Lemma \ref{l41} it suffices to show that
each $q\in\vv\setminus(\overline{\uu}\cup\set{\varphi})$ belongs to $\bb_2^{(2)}$.

Take $q\in\vv\setminus(\overline{\uu}\cup\set{\varphi})$ arbitrarily. 
By Lemma \ref{l27} there exists a word $a_1\cdots  a_m$ with  $m\ge 2$ such that 
\begin{equation*}
\al(q)=(a_1\cdots a_m^+\overline{a_1\cdots a_m^+})^\f\in\vv'\qtq{and} (a_1\cdots a_m)^\f\in\vv',
\end{equation*}
where $\vv'$ is defined in (\ref{eq:v}). 
The proof of Lemma \ref{l43} shows that $q=q_{\c,\d}$, where
\begin{equation*}
\c=\overline{a_1\cdots a_m^+}(a_1\cdots a_m)^\f\in A_q'\qtq{and}\d=0^{2m}(\overline{a_1\cdots a_m})^\f\in A_q'.
\end{equation*}

Set
\begin{equation*}
\d_k=0^{2m}(\overline{a_1\cdots a_m})^k( \overline{a_1\cdots a_m^+}a_1\cdots a_m^+)^\f,\quad k=1,2,\ldots .
\end{equation*}
We claim that $\d_k\in A'_{p}$ for all $k\ge 1$ and $p\in (q,2]$. 
 
It suffices to prove that
\begin{equation}\label{44}
\overline{\al(p)}< \si^i((\overline{a_1\cdots a_m})^k( \overline{a_1\cdots a_m^+}a_1\cdots a_m^+)^\f )< \al(p)\qtq{for all} i\ge 0.
\end{equation}
Since $(a_1\cdots a_m^+ \overline{a_1\cdots a_m^+})^\f=\al(q)\in\vv'$ and $\al(p)>\al(q)$, we infer from Lemmas \ref{l21} and \ref{l25} that
\begin{equation*}
\overline{\al(p)}<\overline{\al(q)}\le \si^i((\overline{a_1\cdots a_m^+}a_1\cdots a_m^+)^\f)\le \al(q)< \al(p)\qtq{for all} i\ge 0.
\end{equation*}
It remains to prove \eqref{44} for $0\le i<mk$. 
Since $(a_1\cdots a_m)^\f\in\vv'$ we have
\begin{equation*}
\overline{a_1\cdots a_{m}^+}< \overline{a_{i+1}\cdots a_m a_1\cdots a_i}<a_1\cdots a_{m}^+\qtq{for all} 0\le i<m,
\end{equation*}
and this implies \eqref{44}.

We have shown that $\d_k\in A'_{p}$ for all  $p>q$ and $k\ge 1$. 
Since $\d_k<\d$,  using the equality $\al(q)=(a_1\cdots a_m^+\overline{a_1\cdots a_m^+})^\f$ it follows that
\begin{align*}
f_{\c,\d_k}(q)&=(1\c)_q+(1\d_k)_q-(1^\f)_q\\
&=(1\c)_q+(10^{2m}(\overline{a_1\cdots a_m})^{k+1}0^\f)_q-(1^\f)_q\\
&<(1\c)_q+(1\d)_q-(1^\f)_q=f_{\c,\d}(q)=0.
\end{align*}
By Lemma \ref{l34} this implies that $f_{\c,\d_k}$ has a unique root $q_{\c,\d_k}$ in $(q,2]$,  and  then it belongs to $\bb_2$.   
Since $\d_k$ strictly increases to $\d$ as $k\ra\f$, we conclude by Lemma \ref{l35} and continuity that  $q_{\c,\d_k}\searrow q$ as $k\ra\f$.

By the same argument as above we can also show that for each fixed $k\ge 1$, 
\begin{equation*}
q_{\c_\ell,\d_k}\in\bb_2,\qtq{and} q_{\c_\ell,\d_k}\nearrow q_{\c,\d_k}\qtq{as} \ell\ra\f,
\end{equation*}
where
\begin{equation*}
\c_\ell:=\overline{a_1\cdots a_m^+}(a_1\cdots a_m)^\ell(a_1\cdots a_m^+\overline{a_1\cdots a_m^+})^\f.
\end{equation*}

Therefore  $q\in\bb_2^{(2)}$. \qed
\end{proof}

Now we apply the above results to prove some theorems stated in the introduction:

\newproof{proof13}{Proof of Theorem \ref{t13}}
\begin{proof13}
We combine Proposition \ref{p44} and Theorem \ref{t11} (i). \qed
\end{proof13} 

\newproof{proof15}{Proof of Theorem \ref{t15}  (i)-(iv)}
\begin{proof15}
(i)-(iii) were established in Lemmas \ref{l45}, \ref{l43} and \ref{l46}, respectively.
\medskip 

(iv) We know that the smallest element $q_s$ of $\bb_2$ is isolated.
The second smallest element $q_f$ of $\bb_2$ belongs to $\vv$ (it is  the smallest element in $\vv$ after $\varphi$), hence it is the smallest  accumulation point of $\bb_2$ by the preceding lemma.
Indeed, $q_f\approx 1.75488$ satisfies $q_f^4=q_f^3+q_f^2+1$. 
Hence $\al(q_f)=(1100)^\f$, so that  $q_f\in\vv\setminus\set{\varphi}$
by Lemma \ref{l25}.  \qed
\end{proof15}

Theorem \ref{t15} (v)-(vi) will be proved in Section \ref{s7} below.

\newproof{proof17}{Proof of Theorem \ref{t17} (i)-(iii)}
\begin{proof17}

(i) and (ii) coincide with Lemmas \ref{l41} and \ref{l46}, respectively. (iii) follows from Lemma \ref{l46} and Theorem \ref{t15} (iv). \qed
\end{proof17}

Theorem \ref{t17} (iv)-(vi) will be proved in Sections \ref{s6} and \ref{s7} below.

\section{Explicit description of unique expansions}\label{s5}

In this section  we give the detailed  description of unique expansions. This will be used to investigate the derived sets $\bb_2^{(j)}$ and $\bb_2^{(\f)}$ in the next sections. 

We recall that  $\overline{\uu}$ is a Cantor set, and its smallest element is the Komornik--Loreti constant $q_{KL}$. 
Furthermore, its complement is the union of countably many disjoint open intervals:
\begin{equation}\label{51}
(1,2]\setminus\overline{\uu}=(1,2)\setminus\overline{\uu}=\bigcup(q_0, q_0^*),
\end{equation}
where $q_0$ runs over $\set{1}\cup(\overline{\uu}\setminus\uu)$ and $q_0^*$ runs over a proper subset $\uu^*$ of $\uu$, formed by the \emph{de Vries--Komornik numbers}. 
The first connected component is $(1,q_{KL})$. 
Each left endpoint $q_0$ is an algebraic integer, and each right endpoint $q_0^*$ is a transcendental number.

Since  $q_0\in\set{1}\cup(\overline{\uu}\setminus\uu)$, we have $\al(q_0)=(a_1\cdots  a_m)^\f$ with a smallest periodic block  $a_1\cdots  a_m$   satisfying
\begin{equation}\label{52}
\overline{a_i\cdots a_m a_1\cdots  a_{i-1}}\le a_1\cdots a_m^+\qtq{and} a_i\cdots a_m^+\overline{a_1\cdots a_{i-1}}\le a_1\cdots  a_m^+ 
\end{equation}
for all {$0<i\le m$}.
This implies that  $\beta(q_0)=a_1\cdots  a_m^+ 0^\f$. 
Here we use the convention $\al(1):=0^\f$ for $q_0=1$, although $0^\f$ is not a $1$-expansion of $1$. 
However $\beta(1)=10^\f$ is a $1$-expansion of $1$ in the natural sense. (We have not defined $q$-expansions for $q=1$.)
 
We will say that the interval $(q_0, q_0^*)$ is the \emph{connected component} of $(1, 2]\setminus\overline{\uu}$ generated by $a_1\cdots  a_m$.   
In particular, the interval $(1, q_{KL})$ is called the {connected component of $(1,2]\setminus\overline{\uu}$ generated by the word $a_1=0$}. 

We recall that $\vv$ is the set of bases $q\in(1,2]$ such that $1$ has a unique doubly infinite $q$-expansion, and that   $\overline{\uu}\subset\vv$. 
Furthermore,  for each connected component $(q_0, q_0^*)$ of $(1,2]\setminus\overline{\uu}$, we have 
\begin{equation*}
\vv\cap(q_0, q_0^*)=\set{q_n: n=1,2,\ldots},
\end{equation*}
where $q_0<q_1<q_2<\cdots<q_0^*$ and $q_n\nearrow q_0^*$ as $n\ra\f$. Therefore
\begin{equation*}
(q_0, q_0^*)\setminus\vv=\bigcup_{n=0}^\f(q_n, q_{n+1}).
\end{equation*}

Let $(q_0, q_0^*)$ be the connected component of $(1, 2]\setminus\overline{\uu}$ generated by $\omega_0^-=a_1\cdots  a_m$.   
We define a sequence of words recursively by the formulas $\omega_0:=a_1\cdots a_m^+$ and
\begin{equation*}
\omega_{n+1}:=\omega_n\overline{\omega_n}^+,\quad n=0,1,\ldots.
\end{equation*}
For example, 
\begin{equation*}
\omega_1=a_1\cdots a_m^+\overline{a_1\cdots a_m}, \qtq{and} \omega_2=a_1\cdots a_m^+\overline{a_1\cdots a_m}\overline{a_1\cdots a_m^+}a_1\cdots  a_m^+.
\end{equation*}
Observe that $\omega_n$ is a word of length $2^n m$. 

We recall from \cite{Komornik_Loreti_2007,DeVries_Komornik_2008,Vries-Komornik-Loreti-2016} the relations
\begin{equation}\label{53}
\alpha(q_n)=(\omega_n^-)^\f=(\omega_{n-1}\overline{\omega_{n-1}})^\f,\qtq{and} \beta(q_n)=\omega_n 0^\f
\end{equation}
for all $n=0,1,\ldots .$
It follows that the infinite sequence
$\al(q_0^*)=\lim_{n\ra\f}\al(q_n)$  begins with
\begin{equation*}
\omega_0\overline{\omega_0}^+\overline{\omega_0}\omega_0\;\overline{\omega_0}\omega_0^-\omega_0\overline{\omega_0}^+.
\end{equation*}
The construction shows that for $q_0^*=q_{KL}$ the quasi-greedy expansion  $\al(q_0^*)$ is the truncated Thue-Morse sequence (see Section \ref{s7}). 

The main purpose of this section is to describe explicitly the difference sets $\uu_{q}'\setminus\uu_{q_0}'$ for all $q\in (q_0,q_0^*]$ for any connected component $(q_0, q_0^*)$  of $(1,2]\setminus\overline{\uu}$.
 
We need the following result from \cite{Kong_Li_2015}:
 
\begin{lemma}\label{l51}
Let $(q_0, q_0^*)$ be a connected component of $(1, 2]\setminus\overline{\uu}$ generated by $\omega_0^-$, and let $(c_i)\in\uu_{q_0^*}'$.
 
\begin{itemize}
\item  If $c_\ell=0$ for some $\ell\ge 1$ and $c_{\ell+1}\cdots  c_{\ell+2^k m}=\omega_k$ for some $k\ge 0$, then the next block of length $2^k m$ is
\begin{equation*}
c_{\ell+2^k m+1}\cdots c_{\ell+2^{k+1}m}=\overline{\omega_k}\qtq{or}  c_{\ell+2^k m+1}\cdots c_{\ell+2^{k+1}m}=\overline{\omega_k}^+.
\end{equation*}
\item Symmetrically, if $c_\ell=1$ for some $\ell\ge 1$ and $c_{\ell+1}\cdots  c_{\ell+2^k m}=\overline{\omega_k}$ for some $k\ge 0$, then the next block of length $2^k m$ is
\begin{equation*}
c_{\ell+2^k m+1}\cdots c_{\ell+2^{k+1}m}= {\omega_k}\qtq{or}  c_{\ell+2^k m+1}\cdots c_{\ell+2^{k+1}m}= {\omega_k}^-.
\end{equation*}
\end{itemize}
\end{lemma}

We also need  \cite[Lemma 4.2]{Kong_Li_2015}:

\begin{lemma}\label{l52}
Let $(q_0, q_0^*)$ be a connected component of $(1, 2]\setminus\overline{\uu}$ generated by $\omega_0^-$. Then for any $n\ge 0$ the word 
  $\omega_n=\theta_1\cdots\theta_{2^n m}$ satisfies 
\begin{equation*}
\overline{\theta_1\cdots\theta_{2^n m-i}}<\theta_{i+1}\cdots \theta_{2^n m}\le \theta_1\cdots\theta_{2^n m-i}
\end{equation*}
for all  ${0\le i < 2^n m}$.
\end{lemma}

Now we prove the main result of this section:

\begin{theorem}\label{t53}
Let $(q_0, q_0^*)$ be a connected component of $(1, 2]\setminus\overline{\uu}$ generated by $\omega_0^-$. Then 
 $\uu_{q_0^*}'\setminus\uu_{q_0}'$ is formed by the sequences of the form
\begin{equation}\label{54}
\omega(\omega_0^-)^{j_0}(\omega_{k_1}\overline{\omega_{k_1}})^{j_1}(\omega_{k_1}\overline{\omega_{k_2}})^{s_2}(\omega_{k_2}\overline{\omega_{k_2}})^{j_2}(\omega_{k_2}\overline{\omega_{k_3}})^{s_3}
\cdots
\end{equation}
and their reflections,
where the word $\omega$ satisfies {$\omega(\omega_k^-)^\f\in\uu_{q_0^*}'$ for all $k\ge 0$}, and the indices $k_r$, $j_r$, $s_r$ satisfy the following conditions:
\begin{align*}
&0\le k_1<k_2<\cdots\qtq{are integers;}\\
&  j_r\in\set{0,1,\cdots}\qtq{or}j_r=\infty\qtq{for each}r;\\
&s_r\in\set{0,1}\qtq{for all}r;\\
&j_r=\infty\Longrightarrow s_u=j_u=0\qtq{for all}u>r.
\end{align*}
\end{theorem} 

\begin{remark}\label{r54}
We may assume in Theorem \ref{t53} that {$s_r+j_{r}+s_{r+1}\ge 1$ whenever $r\ge 2$ and $j_r<\f$. Indeed, in case $s_r+j_{r}+s_{r+1}=0$} the word $\omega_{k_r}$ is missing, and we may renumber the blocks. 
\end{remark}
 
\begin{proof}
First we show that each sequence in $\uu_{q_0^*}'\setminus\uu_{q_0}'$ or its reflection is of the form \eqref{54}.
 
Fix $(c_i)\in\uu_{q_0^*}'\setminus\uu_{q_0}'$ arbitrarily. 
There exists a smallest  integer $N\ge 1$ such that 
\begin{equation}\label{55}
c_N=0\qtq{and} c_{N+1}c_{N+2}\cdots\ge \al(q_0)=(\omega_0^-)^\f
\end{equation} 
or 
\begin{equation*}
c_N=1,\quad c_{N+1}c_{N+2}\cdots\le \overline{\al(q_0)}=(\overline{\omega_0}^+)^\f.
\end{equation*}
Since the  reflection $\overline{(c_i)}$ of $(c_i)$ also belongs to $\uu_{q_0^*}'$, we may assume that \eqref{55} holds.

{
Then setting $\omega_0^-=:a_1\ldots a_m$ we have
\begin{equation}\label{red51}
c_{N+1}\cdots c_{N+m-1}=a_1\cdots a_{m-1},
\end{equation}
and for $0<n<N$ we have
\begin{equation}\label{red52}
\begin{split}
&c_{n+1}\cdots c_{n+m}\le a_1\ldots a_m\qtq{if} c_n=0,\\
&c_{n+1}\cdots c_{n+m}\ge\overline{a_1\ldots a_m}\qtq{if} c_n=1.
\end{split}
\end{equation}

Indeed, \eqref{red51} follows from the relations
\begin{equation*}
\al(q_0)\le c_{N+1}c_{N+2}\cdots<\al(q_0^*),
\end{equation*}
because both $\al(q_0)$ and $\al(q_0^*)$ start with $a_1\cdots a_{m-1}$,
while \eqref{red52} follows from the minimality of $N$ implying 
\begin{equation*}
\begin{split}
&c_{n+1}c_{n+2}\cdots <\al(q_0)\qtq{if} c_n=0,\\
&c_{n+1}c_{n+2}\cdots >\overline{\al(q_0)}\qtq{if} c_n=1,
\end{split}
\end{equation*}
because $\al(q_0)$ begins with $a_1\ldots a_m$.

We claim that the word $\omega=c_1\cdots c_N$ satisfies $\omega(\omega_k^-)^\f\in\uu_{q_0^*}'$ for all $k\ge 0$.
Since $(\omega_k^-)^\f\in\uu_{q_0^*}'$ by Lemma \ref{l22}, it suffices to prove for each $0<n\le N$ the relations
\begin{equation*}
\begin{split}
&c_{n+1}\cdots c_N(\omega_k^-)^\f<\al(q_0^*)\qtq{if} c_n=0,\\
&{c_{n+1}\cdots c_N(\omega_k^-)^\f}>\overline{\al(q_0^*)}\qtq{if} c_n=1.
\end{split}
\end{equation*}
If $n=N$, then $c_N=0$ and $(\omega_k^-)^\f=\al(q_k)<\al(q_0^*)$.
Otherwise the relations follow from \eqref{red52} because each $(\omega_k^-)^\f$ begins with $c_{N+1}\cdots c_{N+m-1}$ by \eqref{red51}, and $\al(q_0^*)$ begins with $a_1\ldots a_{m-1}a_m^+$.
}
 
If $(c_i)=c_1\cdots  c_N(\omega_0^-)^\f$, then it is of the form \eqref{54}. 
Otherwise there is a largest integer $j_0\ge  1$ satisfying $c_{N+1}\cdots c_{N+j_0 m}=(\omega_0^-)^{j_0}$. 
Then by \eqref{55}  {and $(c_i)\in\uu_{q_0^*}'$} we have
\begin{equation*}
 c_{N+j_0m +1}\cdots c_{N+(j_0+1)m}=\omega_0.
\end{equation*}
Let $k_1\ge 0$ be the largest integer such that $c_{N+j_0m+1}\cdots c_{N+j_0m+2^{k_1}m}=\omega_{k_1}$. 

Since $c_{N+j_0 m}=0$ and
\begin{equation*}
c_{N+j_0m+1}\cdots c_{N+j_0m+2^{k_1}m}=\omega_{k_1},
\end{equation*}
by Lemma \ref{l51} and the maximality of $k_1$ the next block of length $2^{k_1}m$ must be
\begin{equation*}
 c_{N+j_0m+2^{k_1}m+1}\cdots c_{N+j_0m+2^{k_1+1}m}=\overline{\omega_{k_1}}.
\end{equation*}
Then Lemma \ref{l51} implies that the next block of length $2^{k_1}m$ is 
\begin{equation*}
c_{N+j_0m+2^{k_1+1}m+1}\cdots c_{N+j_0m+2^{k_1+1}m+2^{k_1}m}=\omega_{k_1}\qtq{or}\omega_{k_1}^-.
\end{equation*}
Hence
\begin{equation*}
c_{N+j_0m+1}\cdots c_{N+j_0m+2^{k_1+1}m+2^{k_1}m}=\omega_{k_1}\overline{\omega_{k_1}}\omega_{k_1}
\end{equation*}
or
\begin{equation*}
c_{N+j_0m+1}\cdots c_{N+j_0m+2^{k_1+1}m+2^{k_1}m}=\omega_{k_1}\overline{\omega_{k_1+1}}.
\end{equation*}
In the first subcase we either have 
\begin{equation}\label{57}
c_{N+j_0m+1}c_{N+j_0m+2}\cdots =(\omega_{k_1}\overline{\omega_{k_1}})^\f,
\end{equation} 
or there exists a largest integer $j_1\ge 1$ and a largest integer $k_2>k_1$ such that 
\begin{equation}\label{58}
c_{N+j_0m+1}\cdots c_{N+j_0m+j_1 2^{k_1+1}m+2^{k_2}m}=(\omega_{k_1}\overline{\omega_{k_1}})^{j_1}\omega_{k_2}.
\end{equation}
In the second subcase there exists a largest integer $k_2>k_1$ such that
\begin{equation}\label{59}
c_{N+j_0m+1}\cdots c_{N+j_0m+2^{k_1}m+2^{k_2}m}=\omega_{k_1}\overline{\omega_{k_2}}.
\end{equation}

Using \eqref{57}--\eqref{59} we see that if $(c_{u+i})$ starts with $\omega_k$ where $k$ is the largest such index, then three possibilities may occur:
\begin{enumerate}[\upshape (i)]
\item $(c_{u+i})=(\omega_k\overline{\omega_k})^{\infty}$;
\item $(c_{u+i})$ starts with   $(\omega_k\overline{\omega_k})^{j_k}\omega_{\ell}$ for some $j_k\ge 1$ and a maximal $\ell>k$;
\item $(c_{u+i})$ starts with $(\omega_k\overline{\omega_k})^{j_k}\omega_k\overline{\omega_{\ell}}$ for some $j_k\ge 0$ and a maximal $\ell>k$.
\end{enumerate}

Repeating the above reasoning in Subcase (ii) for sequences starting with $\omega_\ell$, we have three possibilities:
\begin{enumerate}[\upshape (i)]
\item[(iia)] $(c_{u+i})=(\omega_k\overline{\omega_k})^{j_k}(\omega_{\ell}\overline{\omega_{\ell}})^{\infty}$ for some $j_k\ge 1$;
\item[(iib)] $(c_{u+i})$ starts with
$
(\omega_k\overline{\omega_k})^{j_k}(\omega_\ell\overline{\omega_{\ell}})^{j_{\ell}} {\omega_n}
$
for some $j_k\ge 1, j_{\ell}\ge 1$ and a maximal  $n>\ell$;
\item[(iic)] $(c_{u+i})$ starts with
$
(\omega_k\overline{\omega_k})^{j_k}(\omega_\ell\overline{\omega_{\ell}})^{j_{\ell}}{\omega_{\ell}}\overline{\omega_n}
$
for some $j_k\ge 1, j_{\ell}\ge 0$ and a maximal  $n>\ell$.
\end{enumerate}

Repeating the above reasoning in  Subcase (iii) for sequences starting with $\overline{\omega_{\ell}}$, we have three possibilities again:
\begin{enumerate}[\upshape (i)]
\item[(iiia)] $(c_{u+i})=(\omega_k\overline{\omega_k})^{j_k}\omega_k(\overline{\omega_{\ell}}\omega_{\ell})^{\infty}=(\omega_k\overline{\omega_k})^{j_k}(\omega_k\overline{\omega_{\ell}})(\omega_{\ell}\overline{\omega_{\ell}})^{\infty}$;
\item[(iiib)] $(c_{u+i})$ starts with
\begin{equation*}
(\omega_k\overline{\omega_k})^{j_k}\omega_k(\overline{\omega_{\ell}}\omega_{\ell})^{j_{\ell}}\overline{\omega_n}
=(\omega_k\overline{\omega_k})^{j_k}(\omega_k\overline{\omega_{\ell}})(\omega_{\ell}\overline{\omega_{\ell}})^{j_{\ell}-1}(\omega_{\ell}\overline{\omega_n})
\end{equation*} 
for some $j_k\ge 0, j_{\ell}\ge 1$ and a maximal  $n>\ell$;
\item[(iiic)] $(c_{u+i})$ starts with
\begin{equation*}
(\omega_k\overline{\omega_k})^{j_k}\omega_k(\overline{\omega_{\ell}}\omega_{\ell})^{j_{\ell}}\overline{\omega_{\ell}}\omega_n
=(\omega_k\overline{\omega_k})^{j_k}(\omega_k\overline{\omega_{\ell}})(\omega_{\ell}\overline{\omega_{\ell}})^{j_{\ell}}\omega_n
\end{equation*} 
for some $j_k\ge 0, j_{\ell}\ge 0$ and a maximal  $n>\ell$.
\end{enumerate}

Iterating this reasoning in subcases (iib), (iic), (iiib) and (iiic) we obtain eventually that 
\begin{equation*} 
(c_i)= c_1\cdots  c_N(\omega_0^-)^{j_0}(\omega_{k_1}\overline{\omega_{k_1}})^{j_1}(\omega_{k_1}\overline{\omega_{k_2}})^{s_2}(\omega_{k_2}\overline{\omega_{k_2}})^{j_2}(\omega_{k_2}\overline{\omega_{k_3}})^{s_3}
\cdots
\end{equation*}
with $k_r, j_r, s_r$ as specified in the statement of the theorem. 

{
Now we prove that, conversely, each sequence of the form \eqref{54} belongs to $\uu_{q_0^*}'\setminus\uu_{q_0}'$. Take a sequence $(c_i)$ of form \eqref{54}.  By Lemma \ref{l22} it suffices to prove that 
\begin{equation}\label{eq:*1qqq}
\begin{split}
c_{n+1}c_{n+2}\ldots& <\al(q_0^*)\qtq{whenever} c_n=0,\\
c_{n+1}c_{n+2}\ldots& >\overline{\al(q_0^*)}\qtq{whenever} c_n=1.
\end{split}
\end{equation}
Write $c_1\ldots c_N=\omega$ and $\omega_0^-=a_1\ldots a_m$. 
We distinguish three cases.

\emph{First case: $0<n<N$.} 
Since $(\omega_k^-)^\f$ is strictly increasing as $k\ra\f$, there exists a large integer $k$ such that 
\begin{align*}
&c_{n+1}c_{n+2}\ldots\le c_{n+1}\ldots c_{N}(\omega_k^-)^\f<\al(q_0^*)\qtq{if} c_n=0,\\
&c_{n+1}c_{n+2}\ldots\ge c_{n+1}\ldots c_N (\omega_0^-)^\f>\overline{\al(q_0^*)}\qtq{if} c_n=1;
\end{align*}
the second inequality of each line follows from the relations $\omega(\omega_k^-)^\f\in\uu_{q_0^*}'$ for all $k\ge 0$. 

\emph{Second case: $N\le n<N+j_0 m$.} 
Writing  $\omega_0^-=:a_1\ldots a_m$ again, we observe that 
\begin{equation*}
c_{N+j_0 m+1}\cdots c_{N+j_0 m+m}=\omega_0=a_1\ldots a_m^+,
\end{equation*}
and therefore 
\begin{equation}\label{512qqq}
c_{n+1}\cdots c_{n+m}=a_{i+1}\ldots a_ma_1\ldots a_i\qtq{for some} 0\le i<m.
\end{equation}

We infer from Lemma \ref{l52} that
\begin{equation*}
\overline{a_1\ldots a_{m-i}}\le a_{i+1}\ldots a_m<a_1\ldots a_{m-i}\qtq{for all} 0< i<m.
\end{equation*}
Hence 
\begin{equation*}
a_{i+1}\ldots a_ma_1\ldots a_i <a_1\ldots a_m^+=\omega_0\qtq{for all} 0\le i<m
\end{equation*}
and
\begin{equation*}
a_{i+1}\ldots a_m a_1\ldots a_i\ge\overline{a_1\ldots a_m}>\overline{a_1\ldots a_m^+}=\overline{\omega_0}\qtq{for all} 0< i<m.
\end{equation*}
Since $\al(q_0^*)$ begins with $\omega_0$, in view of \eqref{512qqq} they imply \eqref{eq:*1qqq}.
}
{\emph Third case: $n\ge N+j_0 m$.}   
Write $\omega_k:=\theta_1\cdots \theta_{2^k m}$ for $k=0,1,\ldots .$  
By definition $\omega_k$ is a prefix of $\al(q_0^*)$ for all $k\ge 0$.   
By Lemma  \ref{l52} the following strict inequalities hold for all $k\ge 0$ and $0<i\le 2^k m$:  
\begin{align*}
&\theta_{i+1}\cdots \theta_{2^k m}\overline{\theta_1\cdots \theta_i}<\theta_1\cdots \theta_{2^k m}\qtq{and}
\theta_{i+1}\cdots \theta_{2^k m}>\overline{\theta_1\cdots \theta_{2^k m-i}}.
\end{align*}
Since $\omega_k$ is a prefix of $\omega_\ell$ whenever $k<\ell$,  {it follows that $c_{n+1}c_{n+2}\ldots$ satisfies \eqref{eq:*1qqq}.
}

Furthermore, we note that $(\omega_0^-)^\f\notin\uu_{q_0}'$ and that the words $\omega_k$ are forbidden in $\uu_{q_0}'$ for all $k\ge 0$. 
We conclude that all sequences of the form \eqref{54} belong to $\uu_{q_0^*}'\setminus\uu_{q_0}'$. 
By symmetry this completes the proof.  \qed
\end{proof}

Now we recall that $A_q'$ is the  set of sequences $(c_i)\in \uu_{q}'$ with $c_1=0$. 
Let $(q_0, q_0^*)$ be a connected component of $(1,2]\setminus\overline{\uu}$ generated by $\omega_0^-$ and write $(q_0, q_0^*)\cap\vv=\set{q_n: n\ge 1}$ as above. Then for each $n\ge 0$ we have $\al(q_n)=(\omega_{n}^-)^\f=(\omega_{n-1}\overline{\omega_{n-1}})^\f$. 
We are going to describe $A_q'\setminus A_{q_0}'$ for any $q\in(q_n, q_{n+1}]$.
This is based on  the following lemma:

\begin{lemma}\label{l55}
Let $(q_0, q_0^*)$ be a connected component of $(1, 2]\setminus\overline{\uu}$ generated by $\omega_0^-=a_1\ldots a_m$. Suppose $(q_0, q_0^*)\cap\vv=\set{q_n: n\ge 1}$. Then for any $q\in(q_n, q_{n+1}]$ the words $\omega_n$ and its reflection $\overline{\omega_n}$ are forbidden in the language of $\uu_q'$.
\end{lemma}

\begin{proof}
Since $\uu_q'$ is reflection invariant, it suffices to prove that $\omega_n$ is forbidden. 

Assume on the contrary that there exists a sequence $(c_i)\in\uu_q'$ with $c_{N+1}\cdots c_{N+2^n m}=\omega_n$. 
Since
$
\al(q)\le \alpha(q_{n+1})=(\omega_n\overline{\omega_n})^\f,
$
it follows that $c_N=0$, and then  $c_{N+1}c_{N+2}\cdots <(\omega_n\overline{\omega_n})^\f$ by Lemma \ref{l22}. 
This implies that 
\begin{equation}\label{510}
c_{N+2^n m+1}\cdots c_{N+2^{n+1}m}\le \overline{\omega_n}.
\end{equation}
On the other hand, since $c_{N+2^n m}=1$, we have $c_{N+2^n m+1}\cdots c_{N+2^{n+1} m}>(\overline{\omega_n}\omega_n)^\f$ by Lemma \ref{l22}. 
Combining with \eqref{510} this yields that 
\begin{equation*}
c_{N+{2^n}m+1}\cdots c_{N+2^{n+1}m}=\overline{\omega_n}.
\end{equation*}

A similar argument shows that if $c_{u+1}\cdots c_{u+2^n m}=\overline{\omega_n}$, then the next block of length $2^n m$ is $\omega_n$.

Iterating the above reasoning we conclude that if $(c_i)$ has a word $\omega_n$ then it will eventually end with $(\omega_n\overline{\omega_n})^\f=\al(q_{n+1})$. 
This contradicts our initial assumption $(c_i)\in\uu_q'$. \qed
\end{proof}

Theorem \ref{t53} and Lemma \ref{l55} imply the following

\begin{corollary}
\label{c56}
Let $(q_0, q_0^*)$  be a connected component of $(1, 2]\setminus\overline{\uu}$ generated by $\omega_0^-$, and let $(q_0, q_0^*)\cap\vv=\set{q_n: n\ge 1}$ with $q_1<q_2<\cdots$. If $q\in(q_n, q_{n+1}]$ for some $n\ge 1$, then all elements of $A_q'\setminus A_{q_0}'$ are given by 
\begin{equation*}
\omega(\omega_0^-)^{j_0}(\omega_{k_1}\overline{\omega_{k_1}})^{j_1}(\omega_{k_1}\overline{\omega_{k_2}})^{s_2}(\omega_{k_2}\overline{\omega_{k_2}})^{j_2}\cdots(\omega_{k_{m-1}}\overline{\omega_{k_m}})^{s_m}(\omega_{k_m}\overline{\omega_{k_m}})^{j_m},
\end{equation*}
where the \emph{nonempty} word $\omega$ satisfies {$\omega(\omega_{k}^-)^\f\in A_q'$ for all $0\le k<n$}, and 
\begin{align*}
&0\le k_1<k_2<k_3<\cdots<k_m<n\qtq{are integers;}\\
& j_r\in\set{0,1,\cdots} \qtq{for}r=0,\ldots,m-1,\qtq{and} j_m=\f;\\
&s_r\in\set{0,1}\qtq{for}r=2,\ldots,m.
\end{align*}
\end{corollary}

\begin{remark}\label{r57}\mbox{}

\begin{itemize}
\item { We emphasize that the initial word $\omega$ cannot be empty.
For otherwise the  corresponding sequence might begin with the digit $1$ for $j_0=0$, while the sequences in $A_q'$ start with $0$.}

\item If $(q_0, q_0^*)=(1, q_{KL})$, then the initial word $\omega$ is simply of the form $0^j$ because $(\omega_0^-)^\f=0^\f$.
\end{itemize}
\end{remark}

\section{Proof of Theorem \ref{t17} (iv)}\label{s6}

Based on the results of the preceding section, in this section we investigate the derived sets $\bb_2^{(1)},\bb_2^{(2)},\ldots .$

Since $\overline{\uu}\subset\bb_2^{(\f)}$ by Lemma \ref{l41}, it suffices to investigate the derived sets $\bb_2^{(j)}$ in each connected component $(q_0, q_0^*)$ of  $(1,2]\setminus \overline{\uu}$. 

In the following we consider an arbitrary connected component  $(q_0, q_0^*)$ of  $(1,2]\setminus \overline{\uu}$ generated by $\omega_0^-=a_1\cdots  a_m$ satisfying \eqref{52}, and we write $(q_0, q_0^*)\cap\vv=\set{q_n: n\ge 1}$ with $q_1<q_2<\cdots$ as usual, so that
\begin{equation*}
(q_0, q_0^*)=\bigcup_{n=0}^\f(q_n, q_{n+1}].
\end{equation*}
By Corollary \ref{c56} and Lemma \ref{l32} for each $q\in(q_n, q_{n+1}]\cap\bb_2$ there exist two words $\omega, \tilde\omega$ satisfying {$\omega(\omega_k^-)^\f, \tilde\omega(\omega_k^-)^\f\in A_q'$ for all $0\le k<n$}, and a pair of vectors
\begin{equation*}
(\k,\s,\j)\qtq{and}(\tilde\k,\tilde\s,\tilde\j)
\end{equation*}
satisfying the statements in Corollary \ref{c56} such that $q=q_{\omega, \k,\s,\j;\tilde\omega,\tilde\k,\tilde\s,\tilde\j}$ with an obvious notation. 
We mention that a $q\in\bb_2$ may have multiple representations (see Remark \ref{r42}). 

We have the following result:

\begin{lemma}\label{l61}
Let $(q_0, q_0^*)$ be a connected component of $(1, 2]\setminus\overline{\uu}$ generated by $\omega_0^-$ and write $(q_0, q_0^*)\cap\vv=\set{q_n: n\ge 1}$ with $q_1<q_2<\cdots.$
Suppose  
\begin{equation*}
q=q_{\omega,\k,\s,\j; \tilde\omega,\tilde\k,\tilde\s,\tilde\j}\in\bb_2\cap(q_n,q_{n+1}]
\end{equation*}
{for some $n\ge 0$}, where the   words $\omega,\tilde\omega$ satisfy 
{$\omega(\omega_k^-)^\f\in A_q',   \tilde\omega(\omega_k^-)^\f\in A_q'$ for all $0\le k<n$},  and the  vectors
\begin{align*}
&(\k,\s,\j)=(k_1,\ldots, k_{m_1}; s_2,\ldots, s_{m_1}; j_0, \cdots, j_{m_1-1}, \f),\\
&(\tilde \k,\tilde \s,\tilde \j)=(\tilde k_1,\ldots, \tilde k_{m_2}; \tilde s_2,\ldots, \tilde s_{m_2}; \tilde j_0, \cdots, \tilde j_{m_2-1}, \f)
\intertext{satisfy for some $j\ge 0$ the conditions}
&m_1\ge 0, \quad  m_2\ge 0 \qtq{and} 0\le  (k_{m_1}+1)+(\tilde k_{m_2}+1)\le 2n-j.
\end{align*}
(Here we use the convention that $k_{m_1}:=-1$ if $m_1=0$ and $\tilde k_{m_2}=-1$ if $m_2=0$.)
Then $q\in(\bb_2\cap(q_n, q_{n+1}])^{(j)}$.

Moreover, in case $j\ge 1$ and $q\in\bb_2^{(j)}\cap (q_n, q_{n+1})$ then there exist two sequences $(p_u), (r_u)$ in $\bb_2^{(j-1)}\cap(q_n, q_{n+1})$ such that $p_u\nearrow q$ and $r_u\searrow q$ as $u\ra\f$. 
\end{lemma}

\begin{proof}
For $j=0$ the lemma follows from Corollary \ref{c56} and Lemma \ref{l32}.
Proceeding by induction on $j$, assume that the assertion is true for some $0\le j<2n$, and consider $j+1$.

Take a point  $q_{\omega, \k,\s,\j;\tilde \omega, \tilde \k, \tilde \s, \tilde \j}\in\bb_2\cap(q_n,q_{n+1}]$ such that {$\omega(\omega_k^-)^\f, \tilde\omega(\omega_k^-)^\f\in A_q'$ for all $0\le k<n$}, and 
\begin{align*}
(\k,\s,\j)&=(k_1,\ldots, k_{m_1}; s_2,\ldots, s_{m_1}; j_0, \cdots, j_{m_1-1}, \f),\\
(\tilde \k,\tilde \s,\tilde \j)&=(\tilde k_1,\ldots, \tilde k_{m_2}; \tilde s_2,\ldots, \tilde s_{m_2}; \tilde j_0, \cdots, \tilde j_{m_2-1}, \f)
\end{align*}
satisfy
\begin{equation}\label{61}
0\le (k_{m_1}+1)+(\tilde k_{m_2}+1)\le 2n-(j+1).
\end{equation} 
We will show that $q_{\omega,\k,\s,\j; \tilde\omega, \tilde\k,\tilde\s,\tilde\j}\in(\bb_2\cap(q_n,q_{n+1}])^{(j+1)}$ by distinguishing the cases
\begin{equation*}
q_{\omega,\k,\s,\j; \tilde\omega, \tilde\k,\tilde\s,\tilde\j}\in (q_n, q_{n+1})
\qtq{and} 
q_{\omega,\k,\s,\j; \tilde\omega, \tilde\k,\tilde\s,\tilde\j}=q_{n+1}.
\end{equation*}

\emph{First case:} $q_{\omega,\k,\s,\j; \tilde\omega, \tilde\k,\tilde\s,\tilde\j} \in(q_n, q_{n+1})$. Since $-1\le k_{m_1}\le n-1$ and $-1\le \tilde k_{m_2}\le n-1$, \eqref{61} implies that $k_{m_1}<n-1$ or $\tilde k_{m_2}<n-1$. 
Without loss of generality we assume $k_{m_1}<n-1$, and we consider for each $u\ge 1$ the vectors 
\begin{align*}
(\k^+,\s^-,\j_u)&:=(k_1,\ldots, k_{m_1}, k_{m_1}+1; s_2,\ldots, s_{m_1}, 0; j_0, \cdots, j_{m_1-1}, u, \f)
\intertext{and}
(\k^+,\s^+,\j_u)&:=(k_1,\ldots, k_{m_1}, k_{m_1}+1; s_2,\ldots, s_{m_1}, 1; j_0, \cdots, j_{m_1-1}, u, \f).
\end{align*}
Applying the induction hypothesis, it follows by continuity that
\begin{equation*}
q_{\omega,\k^+,\s^-, \j_u;\tilde\omega, \tilde \k,\tilde \s, \tilde \j}, \quad q_{\omega,\k^+,\s^+, \j_u; \tilde\omega, \tilde \k,\tilde \s, \tilde \j}\in\bb_2^{(j)}\cap(q_n, q_{n+1}).
\end{equation*}
for all sufficiently large $u$.
Writing
{\begin{align*}
\omega^{\k,\s,\j}&:= (\omega_0^-)^{j_0}(\omega_{k_1}\overline{\omega_{k_1}})^{j_1}(\omega_{k_1}\overline{\omega_{k_2}})^{s_2}(\omega_{k_2}\overline{\omega_{k_2}})^{j_2}\cdots(\omega_{k_{{m_1}-1}}\overline{\omega_{k_{m_1}}})^{s_{m_1}}
(\omega_{k_{m_1}}\overline{\omega_{k_{m_1}}})^\infty,\\
\omega^{\k^+,\s^-,\j_u}&:= (\omega_0^-)^{j_0}(\omega_{k_1}\overline{\omega_{k_1}})^{j_1}(\omega_{k_1}\overline{\omega_{k_2}})^{s_2}(\omega_{k_2}\overline{\omega_{k_2}})^{j_2}\cdots\\
&\hspace{4cm}(\omega_{k_{{m_1}-1}}\overline{\omega_{k_{m_1}}})^{s_{m_1}}
(\omega_{k_{m_1}}\overline{\omega_{k_{m_1}}})^u(\omega_{k_{m_1}+1}\overline{\omega_{k_{m_1}+1}})^\f,\\
\omega^{\k^+,\s^+,\j_u}&:= (\omega_0^-)^{j_0}(\omega_{k_1}\overline{\omega_{k_1}})^{j_1}(\omega_{k_1}\overline{\omega_{k_2}})^{s_2}(\omega_{k_2}\overline{\omega_{k_2}})^{j_2}\cdots\\
&\hspace{3cm}(\omega_{k_{{m_1}-1}}\overline{\omega_{k_{m_1}}})^{s_{m_1}}
(\omega_{k_{m_1}}\overline{\omega_{k_{m_1}}})^u(\omega_{k_{m_1}}\overline{\omega_{k_{m_1}+1}})(\omega_{k_{m_1}+1}\overline{\omega_{k_{m_1}+1}})^\f
\end{align*}}
for convenience, we have
\begin{align*}
&\omega^{\k^+,\s^-,\j_1}>\omega^{\k^+,\s^-,\j_2}>\omega^{\k^+,\s^-,\j_3}> \cdots>\omega^{\k,\s,\j}
\intertext{and}
&\omega^{\k^+,\s^+,\j_1}<\omega^{\k^+,\s^+,\j_2}<\omega^{\k^+,\s^+,\j_3}<\cdots<\omega^{\k,\s,\j}.
\end{align*}
Therefore, using Lemma \ref{l35} we obtain that for sufficiently large $u$ the sequences
\begin{equation*}
(p_u):=(q_{\omega,\k^+,\s^-, \j_u;\tilde\omega,\tilde \k,\tilde \s, \tilde \j})
\qtq{and}
(r_u):=(q_{\omega,\k^+,\s^+, \j_u;\tilde\omega,\tilde \k,\tilde \s, \tilde \j})
\end{equation*} 
belong to $\bb_2^{(j)}\cap(q_n, q_{n+1})$ and
\begin{equation*}
p_u\nearrow q_{\omega, \k,\s,\j; \tilde\omega, \tilde \k, \tilde \s, \tilde \j},
\quad r_u\searrow q_{\omega, \k,\s,\j; \tilde\omega, \tilde \k, \tilde \s, \tilde \j}.
\end{equation*}  
In particular, $q_{\omega, \k,\s,\j; \tilde\omega, \tilde \k, \tilde \s, \tilde \j}\in(\bb_2\cap(q_n, q_{n+1}])^{(j+1)}$.

\emph{Second case:} $q_{\omega,\k,\s,\j; \tilde\omega, \tilde\k,\tilde\s,\tilde\j}=q_{n+1}$. We obtain similarly to the preceding case that there exists a sequence $(p_u)\in(\bb_2\cap(q_n, q_{n+1}])^{(j)}$ satisfying $p_u\nearrow q_{\omega,\k,\s,\j; \tilde\omega, \tilde\k,\tilde\s,\tilde\j}$ as $u\ra\f$. 
Therefore, $q_{\omega,\k,\s,\j; \tilde\omega, \tilde\k,\tilde\s,\tilde\j}\in (\bb_2\cap(q_n, q_{n+1}])^{(j+1)}$ as required. \qed
\end{proof}

\begin{proposition}\label{p62}
Let $(q_0, q_0^*)$ be  a connected component of $(1, 2]\setminus\overline{\uu}$ generated by $\omega_0^-=a_1\ldots a_m$ and write $(q_0, q_0^*)\cap\vv=\set{q_n: n\ge 1}$ with $q_1<q_2<\cdots$. Then 
\begin{equation*}
\bb_2^{(n)}\cap(q_n, q_{n+1})\ne\emptyset\qtq{for all} n\ge 2. 
\end{equation*}
\end{proposition}

\begin{remark}\label{r63}
The assumption $n\ge 2$ cannot be omitted.
Indeed, for $(q_0,q_0^*)=(1,q_{KL})$ we know that $\bb_2\cap(q_0, q_2)=\set{q_s}$.
Since $q_s>q_1$, this implies that 
\begin{equation*}
\bb_2\cap(q_0, q_1)=\bb_2^{(1)}\cap(q_1, q_2)=\emptyset.
\end{equation*}
\end{remark}

\begin{proof}
Let $n\ge 2$, and set
\begin{equation*}
\omega=0^{2^n m}\overline{\omega_0},\quad  \tilde\omega=\overline{\omega_0},
\end{equation*}
and 
\begin{align*}
&(\k,\s,\j)=(0;\emptyset;1,\f),\quad (\tilde\k,\tilde\s,\tilde\j)=(0,1,\ldots,n-2; 0,\ldots,0; 1,\ldots,1,\f).
\end{align*}
Note that $\omega_0=a_1\cdots  a_m^+$.  {Take $q\in(q_n, q_{n+1})$ with $n\ge 2$. Then $\al(q)$ begins with $a_1\ldots a_m^+\overline{a_1\ldots a_m}$. Pick $0\le k<n$. 
Using Lemma \ref{l52} we see that  
\begin{equation*}
\sigma^i(\overline{a_1\cdots  a_m^+}(\omega_k^-)^\f)<a_1\cdots  a_m^+\overline{a_1\cdots  a_m}\ldots=\al(q)
\end{equation*}
and 
\begin{equation*}
\sigma^i(\overline{a_1\cdots  a_m^+}(\omega_k^-)^\f)>\overline{a_1\cdots  a_m^+}a_1\cdots  a_m\ldots=\overline{\al(q)}
\end{equation*}
for all $i\ge 0$.
Hence $\overline{\omega_0}(\omega_k^-)^\f\in A_q'$, and therefore $\omega(\omega_k^-)^\f, \tilde\omega(\omega_k^-)^\f\in A_q'$ for all $0\le k<n$.} 
It follows from Corollary \ref{c56} that the sequences
\begin{align*}
&\c:=\omega\omega^{\k,\s,\j}=0^{2^n m}\overline{\omega_0}\omega_0^-(\omega_0\overline{\omega_0})^\f
\intertext{and}
&\d:=\tilde\omega\omega^{\tilde\k,\tilde\s,\tilde\j}=\overline{\omega_0}\omega_0^-(\omega_0\overline{\omega_0})\cdots(\omega_{n-3}\overline{\omega_{n-3}})(\omega_{n-2}\overline{\omega_{n-2}})^\f
\end{align*}
also belong to $A_q'$. 
We are going to prove that $f_{\c,\d}(q_n)<0$ and $f_{\c,\d}(q_{n+1})>0$.
It will then follow by Lemmas \ref{l32} and \ref{l34} that $f_{\c,\d}$ has a {(unique)} zero $q_{\omega,\k,\s,\j; \tilde\omega,\tilde\k,\tilde\s,\tilde\j}$ in $(q_n, q_{n+1})$, and then Lemma \ref{l61} will imply that $q_{\omega,\k,\s,\j; \tilde\omega,\tilde\k,\tilde\s,\tilde\j}\in\bb_2^{(n)}\cap(q_n, q_{n+1})$. 

It remains to prove the two inequalities.
Since $\overline{\omega_0}\omega_0^-=\overline{\omega_1}$ and $\omega_k\overline{\omega_k}=\omega_{k+1}^-$ for all $k\ge 0$, we have
\begin{equation}\label{62}
\begin{split}
f_{\c,\d}(q)&=(1\c)_q-(0\overline{\d})_q\\
&=(10^{2^n m}\overline{\omega_0}\omega_0^-(\omega_0\overline{\omega_0})^\f)_q-(0\omega_0\overline{\omega_0}^+(\overline{\omega_0}\omega_0)\cdots(\overline{\omega_{n-3}}\omega_{n-3})(\overline{\omega_{n-2}}\omega_{n-2})^\f)_q\\
&=(10^{2^n m}\overline{\omega_1}(\omega_1^-)^\f)_q-(0\omega_1\overline{\omega_1}^+\, \overline{\omega_2}^+\cdots\overline{\omega_{n-2}}^+\,(\overline{\omega_{n-1}}^+)^\f)_q\\
&\quad \vdots\\
&=(10^{2^n m}\overline{\omega_1}(\omega_1^-)^\f)_q-(0\omega_n(\overline{\omega_{n-1}}^+)^\f)_q
\end{split}
\end{equation}
for all $q$.
Since $(\omega_n0^\f)_{q_n}=1$ and the word $\omega_n$ is of length $2^n m$, we infer from \eqref{62} that 
\begin{equation*}
f_{\c,\d}(q_n)=(0^{2^n m+1}\overline{\omega_1}(\omega_1^-)^\f)_{q_n}-(0^{2^n m+1}(\overline{\omega_{n-1}}^+)^\f)_{q_n}<0,
\end{equation*}
because 
\begin{equation*}
\overline{\omega_1}(\omega_1^-)^\f<(\overline{\omega_{n-1}}^+)^\f
\end{equation*}
for all $n\ge 2$.
Indeed, $(\overline{\omega_{n-1}}^+)^\f$ starts with
\begin{equation*}
\begin{split}
&\overline{\omega_1}^+>\overline{\omega_1}\qtq{if}n=2,\\
&\overline{\omega_1}\omega_1>\overline{\omega_1}\omega_1^-\qtq{if}n=3,\\
&\overline{\omega_1}\omega_1^-\omega_1>\overline{\omega_1}(\omega_1^-)^2\qtq{if}n\ge 4.
\end{split}
\end{equation*}

Next we observe that $(\omega_n\overline{\omega_n}^+ 0^\f)_{q_{n+1}}=1$. 
Therefore, using \eqref{62} we have 
\begin{align*}
f_{\c,\d}(q_{n+1})&=(10^{2^n m}\overline{\omega_1}(\omega_1^-)^\f)_{q_{n+1}}-(0\omega_n(\overline{\omega_{n-1}}^+)^\f)_{q_{n+1}}\\
&=(0^{2^n m+1}\overline{\omega_1}(\omega_1^-)^\f)_{q_{n+1}}+(0^{2^nm+1}\overline{\omega_n}^+ 0^\f)_{q_{n+1}}-(0^{2^nm+1}(\overline{\omega_{n-1}}^+)^\f)_{q_{n+1}}\\
&=q^{-2^nm-1}[(\overline{\omega_1}(\omega_1^-)^\f)_{q_{n+1}}+(\overline{\omega_n}^+ 0^\f)_{q_{n+1}}-((\overline{\omega_{n-1}}^+)^\f)_{q_{n+1}}].
\end{align*}
 
We observe that
\begin{align*}
&(\overline{\omega_{n-1}}^+)^\f
=(\overline{\omega_{n-2}}\omega_{n-2})^\f
=\overline{\omega_{n-2}}(\omega_{n-1}^-)^\infty
<\overline{\omega_{n-2}}^+ 0^\f\intertext{and}
&\overline{\omega_{n}}^+ 0^\f
=\overline{\omega_{n-2}}\omega_{n-2}^-\omega_{n-2}\overline{\omega_{n-2}}^+0^\f;
\intertext{hence}
&(\overline{\omega_{n-1}}^+)^\f
<\overline{\omega_{n}}^+ 0^\f+0^{2^{n-2}m-1}1 0^\f.
\end{align*}
If $n\ge 3$, then 
\begin{equation}\label{63}
0^{2^{n-2}m-1}10^\f \le \overline{\omega_1}(\omega_1^-)^\f,
\end{equation}
so that
\begin{equation}\label{64}
((\overline{\omega_{n-1}}^+)^\f)_{q_{n+1}}<(\overline{\omega_1}(\omega_1^-)^\f)_{q_{n+1}}+(\overline{\omega_n}^+ 0^\f)_{q_{n+1}},
\end{equation}
and thus $f_{\c,\d}(q_{n+1})>0$ as required.

The relation \eqref{63} and hence the proof of $f_{\c,\d}(q_{n+1})>0$ remains valid if $n=2$ and $\omega_0$ contains at least one zero digit.

In the remaining case we have $n=2$ and 
\begin{equation*}
\omega_0=1^m,\quad \omega_1=1^m0^{m-1}1,\quad \omega_2=1^m0^{m-1}10^m1^m
\end{equation*}
for some $m\ge 1$.

If $m\ge 2$, then 
\begin{align*}
(\overline{\omega_1}(\omega_1^-)^\f)_{q_3}+(\overline{\omega_2}^+ 0^\f)_{q_3}
&=((0^m1^{m-1}0(1^m0^m)^\f)_{q_3}+(0^m1^{m-1}01^m0^{m-1}10^\f)_{q_3}\\
&\ge (0^m10^\f)_{q_3}+(0^m10^\f)_{q_3}\\
&>(0^{m-1}10^\f)_{q_3}\\
&>((0^m1^m)^\f)_{q_3}\\
&=((\overline{\omega_1}^+)^\f)_{q_3},
\end{align*}
so that \eqref{64} holds again.

Finally, if $m=1$, then \eqref{64} takes the form
\begin{equation*}
((01)^\f)_{q_3}<(00(10)^\f)_{q_3}+(00110^\f)_{q_3}.
\end{equation*}
This is equivalent to 
\begin{equation*}
q_3^3<(q_3+1)^2,
\end{equation*}
and this holds because the unique positive root of the polynomial $x^3-(x+1)^2$ is greater than $2$. \qed
\end{proof}

\newproof{proof17-1}{Proof of Theorem \ref{t17} (iv)}
\begin{proof17-1}

It follows from Proposition \ref{p62} that 
\begin{equation*}
\bb_2^{(j)}\cap(q_0, q_0^*)\ne\emptyset\qtq{for all} j\ge 0
\end{equation*}
for all connected components $(q_0, q_0^*)$ of $(1,2]\setminus\overline{\uu}$.   
This implies that all sets $\bb_2, \bb_2^{(1)}, \bb_2^{(2)}, \cdots$ have infinitely many accumulation points in $(q_0, q_0^*)$.  \qed
\end{proof17-1}

\section{Proof of Theorem \ref{t15} (v)-(vi) and Theorem \ref{t17} (v)-(vi)}\label{s7}

In this section we focus on the derived sets of $\bb_2$ in the first connected component $(1, q_{KL})$ of $(1,2]\setminus\overline{\uu}$, generated by $\omega_0^-=0$. 

Note that 
\begin{equation*}
(1,q_{KL})\setminus\vv=\bigcup_{n=0}^{\infty}(q_n,q_{n+1}),
\end{equation*}
where the numbers $q_n$ satisfy 
\begin{equation*}
\alpha(q_n)=(\omega_n^-)^\f=(\tau_1\cdots \tau_{2^n}^-)^\f,\quad \beta(q_n)=\omega_n 0^\f=\tau_1\cdots\tau_{2^n}0^{\infty}\qtq{and}q_n\nearrow q_{KL}\approx 1.78723.
\end{equation*}
Here 
\begin{equation*}
\al(q_{KL})=\tau_1\tau_2\cdots=1101\ 0011\ 0010\ 1101\ \cdots
\end{equation*}
is the truncated Thue--Morse sequence.
We recall that the complete Thue--Morse sequence
$(\tau_i)_{i=0}^\f$ is defined by the formulas $\tau_0:=0$ and
\begin{equation*}
\tau_{2i}:=\tau_i,\quad \tau_{2i+1}=1-\tau_i,\quad i=0,1,\ldots .
\end{equation*} 

The first five elements of the sequence $(q_n)$ are the following:
\begin{equation*}
q_0=1,\quad q_1\approx 1.61803,\quad q_2\approx 1.75488,\quad q_3\approx 1.78460,\quad q_4\approx 1.78721.
\end{equation*}
Note that $q_1=\varphi$ is the Golden Ratio and $q_2=q_f$ is the smallest accumulation point of $\bb_2$ by Theorem \ref{t15}.

Since $A_{q_0}'=\emptyset$ for $q_0=1$, Corollary \ref{c56} and Remark \ref{57} yields  the following description of $A_q'$ for any $q\in(1, q_{KL})$:

\begin{corollary}\label{c71}
If $q\in(q_n, q_{n+1}]\subset(1, q_{KL})$ for some $n\ge 0$, then the elements of $A_q'$ are given by the formula
\begin{equation*}
0^{j_0}(\omega_{k_1}\overline{\omega_{k_1}})^{j_1}(\omega_{k_1}\overline{\omega_{k_2}})^{s_2}(\omega_{k_2}\overline{\omega_{k_2}})^{j_2}\cdots(\omega_{k_{m-1}}\overline{\omega_{k_m}})^{s_m}(\omega_{k_m}\overline{\omega_{k_m}})^{j_m},
\end{equation*}
where
\begin{align*}
&0\le k_1<k_2<k_3<\cdots<k_m<n\qtq{are integers;}\\
&j_0\ge 1,\quad  j_r\in\set{0,1,\cdots} \qtq{for}r=0,\ldots,m-1,\qtq{and} j_m=\f;\\
&s_r\in\set{0,1}\qtq{for}r=2,\ldots,m.
\end{align*}
\end{corollary}

\begin{examples}\label{e72}\mbox{}

\begin{enumerate}[\upshape (i)]
\item If $q\in (q_0,q_1]\approx (1,1.61803]$, then $n=0$, then
\begin{equation*}
A_q'=\set{0^{\infty}}.
\end{equation*}

\item If $q\in (q_1,q_2]\approx (1.61803,1.75488]$, then 
\begin{equation*}
A_q'=\set{0^{\infty},  0^{j_0}(10)^{\infty} \ :\ j_0=1,2,\ldots}.
\end{equation*}

\item If $q\in (q_2,q_3]$, then \begin{equation*}
A_q'=\set{0^{\infty}, 0^{j_0}(10)^{\infty},   0^{j_0}(10)^{j_1}(100)^{s_0}(1100)^\f \ :\ j_0=1,2,\ldots;  j_1=0,1,\ldots; s_0\in\set{0, 1}}.
\end{equation*}
\end{enumerate}
\end{examples}

Using Corollary \ref{c71} we may strengthen a result of Sidorov \cite{Sidorov_2009}:

\begin{proposition}\label{p73}
Each $q\in\bb_2\cap (1,q_{KL})$ is an algebraic integer.
Hence $\bb_2\cap (1,q_{KL})$ is a countable set.
\end{proposition}

\begin{proof}
If $q\in\bb_2\cap (1,q_{KL})$, then there exist $(c_i),(d_i)\in\uu_q'$ satisfying $(1(c_i))_q=(0(d_i))_q$ by Lemma \ref{l31}.
Since the sequences $(c_i),(d_i)$ are eventually periodic by Corollary \ref{c71}, the last equality takes the form
\begin{equation*}
\frac{1}{q}+\frac{f_1(q)}{g_1(q)}=\frac{f_2(q)}{g_2(q)}
\end{equation*}
with suitable polynomials $f_1, f_2, g_1, g_2$ of integer coefficients, satisfying  the conditions
\begin{equation*}
\deg f_1\le (\deg g_1)-2\qtq{and} \deg f_2\le (\deg g_2)-2.
\end{equation*}
Furthermore, $g_1, g_2$ are of the form $q^m(q^n-1)$ with suitable positive integers $m,n$.
Hence $q$ is a zero of the  polynomial
\begin{equation*}
g_1(x)g_2(x)+xf_1(x)g_2(x)-xf_2(x)g_1(x).
\end{equation*}
We conclude by observing that this polynomial is monic by the special form of $g_1, g_2$, mentioned above. \qed
\end{proof}
 
If $q\in(q_n, q_{n+1}]$, then by Corollary \ref{c71} each sequence $(c_i)\in A_q'$ is uniquely determined by a vector 
\begin{equation*}
(\k,\s,\j)=(k_1,k_2,\ldots,k_m;s_2,s_3,\ldots, s_m; j_0,j_1,\ldots, j_{m}) 
\end{equation*}
with the components as in the statement of Corollary \ref{c71}.  Observe that $(\k,\s,\j)$ has $3m$ components if $m\ge 1$, while $(\k,\s,\j)=(\f)$ if $m=0$.  

For convenience we set
\begin{equation*}
\omega^{\k,\s,\j}:=0^{j_0}(\omega_{k_1}\overline{\omega_{k_1}})^{j_1}(\omega_{k_1}\overline{\omega_{k_2}})^{s_2}(\omega_{k_2}\overline{\omega_{k_2}})^{j_2}\cdots(\omega_{k_{m-1}}\overline{\omega_{k_m}})^{s_m}
(\omega_{k_m}\overline{\omega_{k_m}})^\infty.
\end{equation*}
When   $(\k,\s,\j)=(\f)$ we set $\omega^{\k,\s,\j}=0^\f$. 
We infer from Lemma \ref{l32} that for each $q\in\bb_2\cap(q_n, q_{n+1}]$ there exists a pair of vectors
\begin{equation*}
(\k,\s,\j)\qtq{and}(\tilde{\k},\tilde{\s},\tilde{\j})
\end{equation*}
such that setting $\c=\omega^{\k,\s,\j}, \d=\omega^{\tilde\k, \tilde\s,\tilde\j} \in A_{q_{n+1}}'$ the equality $f_{\c,\d}(q)=(1\c)_q+(1\d)_q-(1^\f)_q=0$ holds. 
In this cases we denote  $q$ by $q_{\k,\s,\j;\tilde{\k}, \tilde{s},\tilde{\j}}$. 

{
\begin{remark}\label{r74}
We do not rule out the possibility that an element  $q\in\bb_2\cap(q_n, q_{n+1}]$ has multiple representations by different pairs of vectors $(\k,\s,\j)$ and $(\tilde{\k}, \tilde{\s}, \tilde{\j})$. 
See Lemma \ref{l77} below and Question 6 at the end of the paper.
\end{remark}
}

In order to investigate the topology of $\bb_2\cap(q_n,q_{n+1}]$ we need the following elementary result, probably first published by K\"ursch\'ak \cite{Kurschak1920}:

\begin{lemma}\label{l75}\mbox{}

\begin{enumerate}[\upshape (i)]
\item Every sequence of real numbers has a monotone subsequence.
\item Every sequence $(c_j)$ in $\set{0,1,\ldots,\infty}$ has either a strictly increasing or a constant subsequence.
\end{enumerate}
\end{lemma}

\begin{proof}
(i) See, e.g., \cite[Theorem 1.1 (a), p. 6]{Komornik-book-2016}.
\medskip

(ii) If $(c_j)$ does not have a constant subsequence, then $\infty$ occurs at most finitely many times, and therefore $(c_j)$ has a monotone subsequence $(c_{j_k})$.
Since $(c_{j_k})$ has no constant subsequences either, it has a strictly monotone further subsequence.
It is necessarily increasing because there is no strictly decreasing sequence of nonnegative integers. \qed
\end{proof}

\begin{lemma}\label{l76}\mbox{}
If a sequence
\begin{equation*}
(q_{\k^r,\s^r,\j^r;\tilde{\k}^r,\tilde{\s}^r, \tilde{\j}^r})_{r=1}^{\infty}\subset\bb_2\cap(q_n,q_{n+1}]
\end{equation*}
converges to some point $p\in[q_n, q_{n+1}]$ with $n\ge 0$, then there exists a subsequence in which each component sequence is either constant or strictly increasing.
\end{lemma}

\begin{proof}
Observe that there are only finitely many possibilities for the vectors $\k^r, \s^r, \tilde{\k}^r$ and $\tilde{\s}^r$. Then there exists a subsequence $(r_i)$ such that 
\begin{align*}
(\k^{r_i}, \s^{r_i}, \j^{r_i})&=(k_1,k_2,\ldots, k_{m_1}; s_2, s_3,\ldots, s_{m_1}; j_0^{r_i}, j_1^{r_i},\ldots, j_{m_1-1}^{r_i}, \f);\\
(\tilde\k^{r_i}, \tilde\s^{r_i}, \tilde\j^{r_i})&=(\tilde k_1,\tilde k_2,\ldots, \tilde k_{m_2}; \tilde s_2, \tilde s_3,\ldots, \tilde s_{m_2}; \tilde j_0^{r_i}, \tilde j_1^{r_i},\ldots, \tilde j_{m_2-1}^{r_i}, \f)
\end{align*}
for all $i\ge 1$, where $m_1, m_2\in\set{0,1,\ldots, n-1}$.
(The letters $r_i$ denote superscripts, not exponents.)

Applying Lemma \ref{l75} (ii) repeatedly to each component sequence $(j_u^{r_i})_{i=1}^{\infty}$ and $(\tilde j_v^{r_i})_{i=1}^{\infty}$, we obtain after $m_1+m_2$ steps a subsequence where each component sequence is either constant or strictly increasing. \qed
\end{proof}

We have seen in Remark \ref{r42} that an element $q\in\bb_2$ may have infinitely many representations, i.e.,
we may have $q=q_{\c,\d}$ for infinitely many pairs of sequences $(\c,\d)\in\Omega_q'$. 
For $q\in\bb_2\cap(1,q_{KL})$ we may have only  finitely many representations:

\begin{lemma}\label{l77}\mbox{}
No element $q\in\bb_2\cap(1, q_{KL})$ may be represented by infinitely many pairs of vectors $(\k,\s,\j)$ and $(\tilde{\k}, \tilde{\s}, \tilde{\j})$.
\end{lemma}

\begin{proof}
Assume that $q=q_{\k^r,\s^r,\j^r;\tilde{\k}^r,\tilde{\s}^r, \tilde{\j}^r}\in\bb_2\cap(q_n, q_{n+1}]$ with $n\ge 0$ for a sequence of distinct pairs of vectors  $(\k^r,\s^r,\j^r)$ and $(\tilde{\k}^r,\tilde{\s}^r,\tilde{\j}^r)$.
By the preceding lemma we may assume, by taking a subsequence, that each component sequence is either constant or increasing.
By Lemma \ref{l35} each component sequence has to be constant, contradicting the choice of $ (\k^r,\s^r,\j^r)$ and $(\tilde{\k}^r,\tilde{\s}^r,\tilde{\j}^r)$.  \qed
\end{proof}

Now we investigate the derived sets $\bb_2^{(j)}$ in $(1, q_{KL})$. 
In this case we may improve Lemma \ref{61} by giving a complete characterization of $\bb_2^{(j)}$.

\begin{proposition}\label{p78}
Let $(q_n, q_{n+1})$ be a connected component of $(1,q_{KL})\setminus\vv$ for some $n\ge 0$, and let $j\ge 0$.
A point $q\in\bb_2\cap(q_n,q_{n+1}]$ belongs to $(\bb_2\cap(q_n, q_{n+1}])^{(j)}$ if and only if $q=q_{\k,\s, \j;\tilde{\k}, \tilde{\s}, \tilde{\j}}$
for some pair of vectors
\begin{align*}
(\k,\s,\j)=(k_1,\ldots, k_{m_1}; s_2,\ldots, s_{m_1}; j_0, \cdots, j_{m_1-1}, \f)\intertext{and}
(\tilde \k,\tilde \s,\tilde \j)=(\tilde k_1,\ldots, \tilde k_{m_2}; \tilde s_2,\ldots, \tilde s_{m_2}; \tilde j_0, \cdots, \tilde j_{m_2-1}, \f)
\end{align*}
satisfying the condition
\begin{equation}\label{71}
m_1\ge 0, \quad m_2\ge 0\qtq{and} 0\le (k_{m_1}+1)+(\tilde k_{m_2}+1)\le 2n-j,
\end{equation}
where we write $k_{m_1}:=-1$ if $m_1=0$ and $\tilde k_{m_2}:=-1$ if $m_2=0$.

Furthermore, $(\bb_2\cap(1, q_{n+1}])^{(2n)}=\emptyset$ for all $n\ge 0$.
\end{proposition}

\begin{proof}
The sufficiency follows from Lemma \ref{l61}. 
We prove the necessity by induction on $j$. 

For $j=0$ the necessity follows from Corollary \ref{c71} and Lemma \ref{l32}. 
Now assume that the necessity holds for some $0\le j<2n$, and take $q\in(\bb_2\cap(q_n, q_{n+1}])^{(j+1)}$ arbitrarily. 
Then there exists a sequence  $(p_r)\subset(\bb_2\cap(q_n, q_{n+1}])^{(j)}$  converging to $q$, and we may write
\begin{equation*}
p_r=q_{\k^r,\s^r,\j^r;\tilde{\k}^r,\tilde{\s}^r,\tilde{\j}^r},\quad r=1,2,\ldots .
\end{equation*}
By Lemma \ref{l75} we may even assume  that the vectors  $\k^r,\s^r,\tilde{\k}^r,\tilde{\s}^r$ do not depend on $r$, so that we may write simply $\k,\s,\tilde{\k},\tilde{\s}$, and that each component sequence in $\j^r$ and $\tilde{\j}^r$ is either constant or strictly  increasing.
Since the numbers $p_r$ are different, at least one of these component  sequences is strictly increasing. 
We may thus write
\begin{align*}
(\k,\s,\j^r)=(k_1,\ldots, k_{n_1}; s_2,\ldots, s_{n_1}; j_0^r, \cdots, j_{n_1-1}^r, \f)\intertext{and}
(\tilde \k,\tilde \s,\tilde \j^r)=(\tilde k_1,\ldots, \tilde k_{n_2}; \tilde s_2,\ldots, \tilde s_{n_2}; \tilde j_0^r, \cdots, \tilde j_{n_2-1}^r, \f)
\end{align*}
with 
\begin{equation*}
0\le (k_{n_1}+1)+(\tilde k_{n_2}+1)\le 2n-j
\end{equation*}
by the induction hypothesis.
Letting $r\to\infty$ at least one of the components of  $\j^r$ or $\tilde{\j}^r$ becomes infinite, so that we obtain $q=q_{\k,\s, \j;\tilde{\k}, \tilde{\s}, \tilde{\j}}$ where
\begin{align*}
(\k,\s,\j)=(k_1,\ldots, k_{m_1}; s_2,\ldots, s_{m_1}; j_0, \cdots, j_{m_1-1}, \f)\intertext{and}
(\tilde \k,\tilde \s,\tilde \j)=(\tilde k_1,\ldots, \tilde k_{m_2}; \tilde s_2,\ldots, \tilde s_{m_2}; \tilde j_0, \cdots, \tilde j_{m_2-1}, \f)
\end{align*}
with $m_1+m_2<n_1+n_2$, and hence
\begin{equation*}
k_{m_1}+\tilde k_{m_2}\le k_{n_1}+\tilde k_{n_2}-1\le 2n-2-j-1
\end{equation*}
as claimed.

Note by Theorem \ref{t11} that $(\bb_2\cap(1, q_1])^{(0)}=\emptyset$. Furthermore, since $\min\bb_2^{(2)}=q_2$ and by Theorem \ref{t11} that $\bb_2\cap(1, q_2)=\set{q_s}$, we have $(\bb_2\cap(1,q_2])^{(2)}=\emptyset$. 
If there exists a $q\in (\bb_2\cap(1,q_{n+1}])^{(2n)}$ for some $n\ge 2$, then we infer from \eqref{71} that $m_1=m_2=0$, so that $q=q_{\k,\s,\j; \tilde\k,\tilde\s,\tilde\j}$ with $(\k,\s,\j)=(\f)=(\tilde\k, \tilde\s,\tilde\j)$.
Then
\begin{equation*}
(1\omega^{\k,\s,\j})_2+(1\omega^{\tilde\k,\tilde\s,\tilde\j})_2-(1^\f)_2=(10^\f)_2+(10^\f)_2-(1^\f)_2=0,
\end{equation*}
and hence $q=2$, contradicting our assumption $q\in (1,q_{n+1}]$.  \qed
\end{proof}

{
\begin{remark}\label{r79}
In the proof of Proposition \ref{p78} the selected subsequence $(p_r)$ is monotonic. 
Indeed, if $j_{n}^r$ is the smallest index tending to infinity as $r\ra\f$, then the corresponding sequence $(p_r)$ is strictly increasing if $s_{n+1}=0$, and strictly decreasing if $s_{n+1}=1$.
\end{remark}
}
 
 \newproof{proof15-1}{Proof of Theorem \ref{t15} (v) and (vi)}
\begin{proof15-1}
(v) was established in Proposition \ref{p73}. 
\medskip 

(vi) Since $\bb_2\cap (1, q_{KL})$ is an infinite set, it suffices to prove the density.
Assume on the contrary that there exist an integer $j\ge 0$, a point $q\in\bb_2\cap(1,q_{KL})$ and a closed neighborhood $F$ of $q$ such that $F$ has no isolated points of $\bb_2$.
We may assume that $F\subset (1,q_{KL})$.
Then all points of the closed set $\bb_2\cap F$ are accumulation points, so that it is a non-empty perfect set, and hence it is uncountable by a classical theorem of topology.
This contradicts the countability of $\bb_2\cap(1,q_{KL})$ by Proposition \ref{p73}. \qed
\end{proof15-1}

\newproof{proof17-2}{Proof of Theorem \ref{t17} (v) and (vi)}
\begin{proof17-2}
(v) The inequalities follow from Theorem \ref{t11} (iv) for $j=0$,  from  Theorem \ref{t17} (iii) for $j=1$, and from  Propositions \ref{p62} and  \ref{p78} for $j\ge 2$.

Since the sequence $(\min\bb_2^{(j)})$ is non-decreasing by definition, its limit $L$ satisfies the inequalities
\begin{equation*}
q_{KL}\le L\le \min\bb_2^{(\infty)}.
\end{equation*}
Since $\min\bb_2^{(\infty)}\le q_{KL}$ by Theorem \ref{t17} (i), the claimed limit relations follow.
\medskip 

(vi) We may repeat the proof of Theorem \ref{t15} (vi) by changing $\bb_2$ to $\bb_2^{(j)}$ everywhere. \qed
\end{proof17-2}

\section{Local dimension of $\bb_2$: proof of Theorem \ref{t19}}\label{s8}

Since  $\bb_2\cap(1,q_{KL})$ is at most countable by Proposition \ref{p73},  Theorem \ref{t19} is trivial for $q< q_{KL}$. 
Henceforth we assume that $q\in[q_{KL}, 2]$. 

For the proof we recall the following  result on Hausdorff dimension and H\"{o}lder continuous maps (cf.~\cite{Falconer_1990}):

\begin{lemma}\label{l81}
Let $f: (X, \rho_1)\ra(Y, \rho_2)$ be a H\"{o}lder map between two metric spaces, i.e., there exist two constants $C>0$ and $\lambda>0$ such that
\begin{equation*}
\rho_2(f(x), f(y))\le C\rho_1(x, y)^\lambda
\end{equation*}
for any $x, y\in X$. 
Then $\dim_H X\ge \lambda\dim_H f(X)$.
\end{lemma}

Given $q\in\bb_2\cap [q_{KL}, 2)$ and $\de\in(0, \min\set{q-1, 2-q})$, set
\begin{equation*}
\Omega_q'(\de):=\bigcup_{p\in(q-\de, q+\de)}\set{(\c,\d)\in A'_{p}\times A'_{p}: f_{\c,\d}(p)=0}.
\end{equation*}
The   map
\begin{equation*}
h_{q, \de}:~  \Omega_q'(\de) \ra \bb_2\cap(q-\de, q+\de),\quad (\c, \d) \mapsto  q_{\c,\d}
\end{equation*}
is well-defined and  onto by Lemmas \ref{l32} and \ref{l34}. 
Furthermore, by Lemma \ref{l35} the map $h_{q,\de}$ is strictly decreasing with respect to each of the vectors  $\c$ and $\d$ with $(\c,\d)\in\Omega_q'(\de)$.

We recall that the symbolic space $\set{0, 1}^\f$ is compact with respect to the metric $\rho$  defined by
\begin{equation*}
\rho((c_i), (d_i)):=2^{-\min\set{j\ge 1: c_j\ne d_j}}.
\end{equation*}
In the following lemma we show that the map $h_{q,\de}$ is also H\"{o}lder continuous with respect to the product metric $\rho^2$ on $\Omega_q'(\de)$ defined by
\begin{equation*}
\rho^2((\c, \d), ({\tilde\c}, {\tilde\d})):=\max\set{\rho(\c, {\tilde\c}), \rho(\d, {\tilde\d})}
\end{equation*}
for any $(\c, \d), ({\tilde\c}, {\tilde\d})\in \Omega'_{q}(\de)$.

\begin{lemma}
\label{l82}
Let $q\in\bb_2\cap[q_{KL}, 2)$ and $\de\in(0,  (2-q)/3)$. 
Then the function $h_{q,\de}: \Omega_q'(\de)\ra \bb_2\cap(q-\de, q+\de)$ is H\"{o}lder continuous of order ${\log (q-\de)}/{\log 2}$ with respect to the product metric $\rho^2$ on $\Omega_q'(\de)$.
\end{lemma}

\begin{proof}
Fix $q\in\bb_2\cap[q_{KL}, 2)$ and $0<\de< (2-q)/3$ arbitrarily.  
Let $(\c, \d), ({\tilde\c}, {\tilde\d})\in\Omega_q'(\de)$, and consider their images $q_{\c, \d}=h_{q,\de}((\c,\d))$ and $q_{{\tilde\c}, {\tilde\d}}=h_{q,\de}({\tilde\c}, {\tilde\d})$ in $\bb_2\cap(q-\de, q+\de)$. 
Without loss of generality  we may assume  ${\tilde\c}>\c$ and ${\tilde\d}<\d$. 
Then there exist positive integers $s$ and $t$ such that
\begin{align*}
&\tilde c_1\cdots \tilde c_{s-1} =c_1\cdots c_{s-1}\quad\textrm{and}\quad \tilde c_s>c_s,\\
&\tilde d_1\cdots \tilde d_{t-1} =d_1\cdots d_{t-1}\quad\textrm{and}\quad \tilde d_t<d_t.
\end{align*}
The definitions of $q_{\c,\d}$ and $q_{\tilde\c,\tilde\d}$ imply that
\begin{equation*}
(1^\f)_{q_{\c,\d}}=(1\c)_{q_{\c, \d}}+(1\d)_{q_{\c,\d}}\qtq{and} (1^\f)_{q_{{\tilde\c},{\tilde\d}}}=(1{\tilde\c})_{q_{{\tilde\c}, {\tilde\d}}}+(1{\tilde\d})_{q_{{\tilde\c},{\tilde\d}}}.
\end{equation*}
Hence
\begin{equation}\label{81}
\begin{split}
\left|\frac{1}{q_{\c,\d}-1}-\frac{1}{q_{{\tilde\c},{\tilde\d}}-1}\right|&=\left|(1\c)_{q_{\c,\d}}+(1\d)_{q_{\c,\d}}-(1{\tilde\c})_{q_{{\tilde\c},{\tilde\d}}}-(1{\tilde\d})_{q_{{\tilde\c},{\tilde\d}}}   \right|\\
&\le \left| (1\c)_{q_{\c,\d}}-(1{\tilde\c})_{q_{{\tilde\c},{\tilde\d}}}\right|+\left|  (1\d)_{q_{\c,\d}}- (1{\tilde\d})_{q_{{\tilde\c},{\tilde\d}}}\right|.
\end{split}
\end{equation}
Since $q_{\c,\d}, q_{{\tilde\c}, {\tilde\d}}> q-\de$, we have 
\begin{align*}
\left|(1\c)_{q_{\c,\d}}-(1{\tilde\c})_{q_{{\tilde\c},{\tilde\d}}}\right|&=\left|(1c_1\cdots c_{s-1}0^\f)_{q_{\c,\d}}-(1c_1\cdots c_{s-1}0^\f)_{q_{{\tilde\c},{\tilde\d}}}\right.\\
&\quad \quad\left.+(0^s c_sc_{s+1}\cdots)_{q_{\c,\d}}-(0^s\tilde  c_s\tilde c_{s+1}\cdots)_{q_{{\tilde\c},{\tilde\d}}}\right|\\
&\le\left|(20^\f)_{q_{\c,\d}}-(20^\f)_{q_{{\tilde\c},{\tilde\d}}}\right|\\
&\quad \quad+|(0^s1^\f)_{q_{\c,\d}}|+|(0^s1^\f)_{q_{{\tilde\c}, {\tilde\d}}}|\\
&\le \frac{2|q_{{\tilde\c},{\tilde\d}}-q_{\c,\d}|}{(q-\de)^2}+\frac{2}{(q-\de)^s(q-\de-1)}.
\end{align*}
Similarly, we also have
\begin{equation*}
\left|  (1\d)_{q_{\c,\d}}- (1{\tilde\d})_{q_{{\tilde\c},{\tilde\d}}}\right|\le \frac{2|q_{{\tilde\c},{\tilde\d}}-q_{\c,\d}|}{(q-\de)^2}+\frac{2}{(q-\de)^t(q-\de-1)}.
\end{equation*}
Therefore, using \eqref{81} it follows that
\begin{equation}\label{82}
\begin{split}
 \left|\frac{1}{q_{\c,\d}-1}-\frac{1}{q_{{\tilde\c},{\tilde\d}}-1}\right|
&  \le \frac{4|q_{{\tilde\c},{\tilde\d}}-q_{\c,\d}|}{(q-\de)^2}+\frac{2}{q-\de-1}\left(\frac{1}{(q-\de)^s}+\frac{1}{(q-\de)^t}\right)\\
&  =\frac{4|q_{{\tilde\c},{\tilde\d}}-q_{\c,\d}|}{(q-\de)^2}+\frac{2}{q-\de-1}\left(2^{-s\frac{\log (q-\de)}{\log 2}}+2^{-t\frac{\log (q-\de)}{\log 2}}\right)\\
&  \le \frac{4|q_{{\tilde\c},{\tilde\d}}-q_{\c,\d}|}{(q-\de)^2}+\frac{4}{q-\de-1}\big(\rho^2((\c,\d),({\tilde\c},{\tilde\d}))\big)^{\frac{\log (q-\de)}{\log 2}}.
\end{split}
\end{equation}

On the other hand, since $q_{\c,\d}, q_{\tilde\c,\tilde\d}<q+\de$, we have
\begin{equation*}
\left|\frac{1}{q_{\c,\d}-1}-\frac{1}{q_{{\tilde\c},{\tilde\d}}-1}\right|\ge\frac{|q_{{\tilde\c},{\tilde\d}}-q_{\c,\d}|}{(q+\de-1)^2}.
\end{equation*}
Comining with \eqref{82} we conclude that
\begin{equation*}
|q_{{\tilde\c},{\tilde\d}}-q_{\c,\d}|\le \left(\frac{1}{(q+\de-1)^2}-\frac{4}{(q-\de)^2}\right)^{-1} \frac{4}{q-\de-1} \big(\rho^2((\c,\d),({\tilde\c},{\tilde\d}))\big)^{\frac{\log (q-\de)}{\log 2}}, 
\end{equation*}
where the fractional term on the right hand side is positive since $\de< (2-q)/3$ and $q-\de-1>0$.   \qed
\end{proof}
\newproof{proof19}{Proof of Theorem \ref{t19}}
\begin{proof19}
Let $q\in\bb_2\cap[q_{KL}, 2)$ and $\de\in(0,(2-q)/3)$. 
Using Lemmas \ref{l81} and \ref{l82} we get
\begin{equation}\label{83}
\begin{split}
\dim_H(\bb_2\cap(q-\de,q+\de))
&\le \frac{\log 2}{\log (q-\de)}\dim_H \Omega_q'(\de) \le \frac{\log 2}{\log (q-\de)}\dim_H(\uu_{q+\de}'\times\uu_{q+\de}'),
\end{split}
\end{equation}
where the last inequality follows because $\Omega'_q(\de)$ is a subset of $\uu_{q+\de}'\times\uu_{q+\de}'$. Observe that $(\uu_{q+\de}', \rho)$ is a fractal set whose Hausdorff dimension is given by
\begin{equation*}
\dim_H\uu_{q+\de}'=\frac{h_{top}(\uu_{q+\de}')}{\log 2}.
\end{equation*}

Hence, using \eqref{83} we obtain that
\begin{align*}
\dim_H(\bb_2\cap(q-\de, q+\de))
&\le \frac{\log 2}{\log (q-\de)}\dim_H (\uu_{q+\de}'\times\uu_{q+\de}')
 \\ & \le \frac{2\log 2}{\log (q-\de)}\dim_H\uu_{q+\de}'
=2\frac{h_{top}(\uu_{q+\de}')}{\log (q-\de)}.
\end{align*}
Letting $\de\ra 0$  and applying Lemma \ref{l24} we conclude that
\begin{equation*}
\lim_{\de\ra 0}\dim_H(\bb_2\cap(q-\de, q+\de))\le 2 \frac{h_{top}(\uu_q')}{\log q}=2\dim_H\uu_q. 
\end{equation*}\qed
\end{proof19}

\section{Final remarks}\label{s9}
In this section we  present some open questions.

We have seen that $\bb_2$ has infinitely many isolated points in $(1,q_{KL})$.

{\bf Question 1}. Does $\bb_2$ have any isolated points greater than $q_{KL}$?

By Theorems \ref{t11} and \ref{t19} there exists a smallest number $q_*\in(q_{KL}, \varphi_3]$ such that $[q_*, 2]\subset\bb_2$.

{\bf Question 2}. What is this smallest number $q_*$?

Sidorov proved in \cite{Sidorov_2009}  that $q_s\approx 1.71064$ is an accumulation point of $\bb_{\aleph_0}$, and then Baker proved in \cite{Baker_2015} that $q_s$ is in fact its smallest accumulation point. 

We have shown in Theorem \ref{t15} that $q_f\approx 1.75488$ is the smallest accumulation point of $\bb_2$ { 
and of $\bb_2^{(1)}$. 

{\bf Question 3}. What is the smallest element of $\bb_2^{(j)}$ for $j=3,4,\ldots ?$
}

{\bf Question 4}. What is the smallest accumulation point of $\bb_k$ for $k=3,4,\ldots$? (This question has already been raised in \cite{Baker_Sidorov_2014}.) 

We have shown in Theorem \ref{t15} that $\bb_2$ is closed, and it was proved in \cite{Zou_Kong_2015} by Zou and Kong that $\bb_{\aleph_0}$ is not closed.

{\bf Question 5}. Are the sets $\bb_k$ closed or not for $k=3,4,\ldots$?

Recently, the second author and his coauthors studied in \cite{Kong-Li-Zou-2017} the smallest element of $\bb_2$ for multiple digit sets $\set{0, 1,\ldots, m}$. The ideas and methods in this paper might also be useful to further explore  $\bb_2$ in the multiple digit set case. {For this extension the main difficulty we may encounter is the characterization of $\bb_2$ as in Section \ref{s3}. Furthermore, a generalization of Theorem \ref{t53} for the explicit description of $\uu_{q_0^*}'\setminus\uu'_{q_0}$ is also needed.} 

Each $q\in\overline{\uu}\setminus\uu\subset\bb_2$ has infinitely many representations by Remark \ref{r42}, and this cannot happen for $q\in\bb_2\cap(1, q_{KL})$ by Lemma \ref{l77}.  

{ 
In the last three questions we consider the elements of $\bb_2\cap(1, q_{KL})$.

{\bf Question 6}. Is the representation of each $q\in\bb_2\cap(1, q_{KL})$ unique?

{\bf Question 7}. Is it true that $q_{n+1}$ is isolated from the left in $\bb_2$?

{\bf Question 8}. Is it possible that $q_{n+1}\in\bb_2^{(3)}$ for some $n$?
}

\section*{Acknowledgements}
{The authors thank the anonymous referee for many useful suggestions.} 
The second author was supported by NSFC No.~11401516. 
He would like to thank Professor Yann Bugeaud for his hospitality when he visited Strasbourg University in November, 2016.

\end{document}